\documentclass[11pt]{article}
\usepackage{geometry}
\usepackage[english]{babel}
\usepackage{amssymb}
\usepackage{amsthm}
\usepackage{amsmath}
\usepackage{tikz}
\usepackage{subcaption}
\usepackage{multirow}
\usepackage{tabu}
\usepackage{setspace}
\usepackage{enumerate}

\newtheorem{theorem}{Theorem}
\newtheorem{lemma}[theorem]{Lemma}

\newtheorem{conjecture}[theorem]{Conjecture}

\begin{document}

\onehalfspace

\title{On the Graceful Game}

\author{Luisa Frickes$^1$\and 
Simone Dantas$^1$\and
At\'ilio G.~Luiz$^2$}

\date{}

\maketitle

\begin{center}
{\small 
$^1$ Instituto de Matem\'{a}tica e Estat\'{i}stica, Universidade Federal Fluminense, Niter\'{o}i, Brazil\\
\texttt{frickesluisa@id.uff.br, sdantas@id.uff.br}\\[3mm]
$^2$ Universidade Federal do Cear\'{a}, Campus Quixad\'a, Quixad\'a, Brazil\\
\texttt{gomes.atilio@ufc.br}
}
\end{center}

\begin{abstract}
A graceful labeling of a graph $G$ with $m$ edges consists of labeling the vertices of $G$ with distinct integers from $0$ to $m$ such that, when each edge is assigned as induced label the absolute difference of the labels of its endpoints, all induced edge labels are distinct. Rosa established two well known conjectures: all trees are graceful (1966) and all triangular cacti are graceful (1988). In order to contribute to both conjectures we study graceful labelings in the context of graph games. The Graceful game was introduced by Tuza in 2017 as a two-players game on a connected graph in which the players Alice and Bob take turns labeling the vertices with distinct integers from 0 to $m$. Alice's goal is to gracefully label the graph as Bob's goal is to prevent it from happening. In this work, we study winning strategies for Alice and Bob in complete graphs, paths, cycles, complete bipartite graphs, caterpillars, prisms, wheels, helms, webs, gear graphs, hypercubes and some powers of paths.
\end{abstract}

{\small 

\begin{tabular}{lp{13cm}}
\end{tabular}
}

\section{Introduction}

A \textit{simple graph} $G=(V(G),E(G))$ is an ordered pair where $V(G)$ is a nonempty finite set whose elements are called \textit{vertices} and $E(G)$ is a set of elements called edges, where an \emph{edge} $e\in E(G)$ is a unordered pair of different vertices of $V(G)$, called its \emph{endpoints}. We say that an edge $e$ \emph{connects} two vertices $u,v\in V(G)$ when $e=uv$, and we say $u$ and $v$ are \textit{adjacent} if they are connected by an edge $e$. Two adjacent vertices are also called \emph{neighbors}. As usual, $N(v)$ denotes the set of neighbors of $v\in V(G)$. An \emph{element} of a graph is a vertex or an edge of the graph.

Graph labeling is an area of graph theory that has been attaining a particular importance since the 1960's. The main concern in this area  consists in determining the feasibility of assigning labels to the elements of a graph satisfying certain conditions. Usually, the labels are elements of a set that supports some kind of mathematical operation as, for example, the set of nonnegative integers. 
However, the idea of assigning symbols other than numbers to the elements of a graph is not recent. For example, an old and very studied problem in graph theory is the \emph{vertex coloring problem}, which consists in determining the least number of colors needed to color the vertices of a given graph such that any two adjacent vertices receive distinct colors. Vertex colorings arose in connection with the well known Four Color Conjecture, which remained open for more than 150 years until its solution in 1976~\cite{Appel1977,Appel1977a}.

In the last decades, many contexts have emerged where it is required to label the vertices or the edges of a given graph with numbers. Most of these problems, such as harmonious labelings~\cite{Graham1980} and L(2,1)-labelings~\cite{Griggs1992}, arose naturally from modeling of optimization problems on networks. Formally, given a graph $G$ and a set $L \subset \mathbb{R}$, a \emph{labeling} of $G$ is a vertex labeling $f \colon V(G) \to L$ that induces an edge labeling $g\colon E(G) \to \mathbb{R}$ in the following way: $g(uv)$ is a function of $f(u)$ and $f(v)$, for all $uv \in E(G)$, and $g$ respects some specified restrictions. Depending on the function $g$ chosen, on the restrictions that $g$ is required to satisfy, or even on the chosen subset of labels $L$, many different types of graph labelings can be defined.

One of the oldest and most studied graph labelings is the \textit{graceful labeling}, so named by Golomb~\cite{Golomb1972} and initially introduced by A.~Rosa~\cite{Rosa1967} around 1966. A \textit{graceful labeling} of a graph $G$ with $m$ edges is an injective function $f\colon V(G) \to \{0,1,\ldots,m\}$ such that, when each edge $uv \in E(G)$ is assigned the (induced) label $g(uv) = |f(u)-f(v)|$, all induced edge labels are distinct. A graph $G$ that has a graceful labeling is called \emph{graceful}. Figure~\ref{pic:example} exhibits three graceful graphs.

\begin{figure}[!htb]
\centering
\begin{tikzpicture}[main_node/.style={circle,fill=white,draw,minimum size=.3em,inner sep=2pt]},scale=0.7]
    \node[main_node] (1) at (0,0) {\scriptsize $0$};
    \node[main_node] (2) at (3,0) {\scriptsize $3$};
    \node[main_node] (3) at (1.5,2.5) {\scriptsize $1$};
    \draw (1) -- (2) -- (3) -- (1);
\end{tikzpicture}
\hspace{.5in}
\begin{tikzpicture}[main_node/.style={circle,fill=white,draw,minimum size=.3em,inner sep=2pt]},scale=0.7]
    \node[main_node] (1) at (0,0) {\scriptsize $0$};
    \node[main_node] (2) at (2.5,0) {\scriptsize $1$};
    \node[main_node] (3) at (2.5,2.5) {\scriptsize $5$};
    \node[main_node] (4) at (0,2.5) {\scriptsize $2$};
    \draw (1) -- (2) -- (3) -- (4) -- (1);
    \draw (1) -- (3);
\end{tikzpicture}
\hspace{.5in}
\begin{tikzpicture}[main_node/.style={circle,fill=white,draw,minimum size=.3em,inner sep=2pt]},scale=0.7]
    \node[main_node] (1) at (0,0) {\scriptsize $1$};
    \node[main_node] (2) at (0,2.5) {\scriptsize $3$};
    \node[main_node] (3) at (2.5,1.25) {\scriptsize $0$};
    \node[main_node] (4) at (5,1.25) {\scriptsize $4$};
    \draw (1) -- (2) -- (3) -- (1);
    \draw (3) -- (4);
\end{tikzpicture}
\caption{Three graphs with graceful labelings.}
\label{pic:example}
\end{figure}
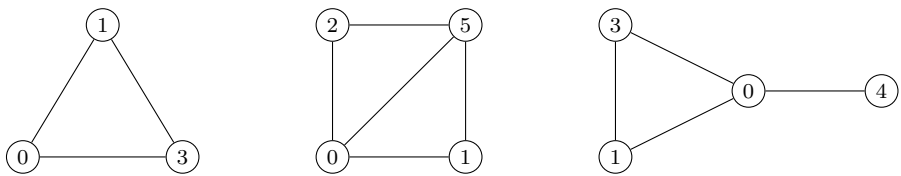

The most studied open problem related to graceful labelings is the \emph{Graceful Tree Conjecture}, which states that all trees are graceful. This conjecture was posed by Rosa~\cite{Rosa1967} in 1966 and, since then, this and many other labeling problems have been studied~\cite{Graham1980,Griggs1992,Rosa1967,Rosa1970}. Another conjecture was posed by Rosa~\cite{Rosa1988} in 1988 that all triangular cacti are graceful. For a comprehensive list of results on graceful labelings, the reader is referred to Gallian's dynamic survey~\cite{Gallian2018}.

From the vast literature of Graph Labeling (more than two thousand papers~\cite{Gallian2018}), it is notorious that labeling problems are usually studied from the perpective of determining whether a given graph has a required labeling or not. An alternative outlook is to analyze labeling problems from the point of view of combinatorial games. The study of combinatorial games is a classical area in both discrete mathematics and game theory~\cite{Berlekamp1982}. One of the main characteristics of these types of games is that there is absolutely no lucky involved, that is, all players have perfect information and involves no chance. In most combinatorial games, two players --- traditionally called Alice and Bob --- alternately select and label vertices or edges (typically one vertex or edge in each step) in a graph $G$ which is completely known for both players.

Despite the rich literature of graph labeling~\cite{Gallian2018}, only a few papers have been published on labeling games~\cite{Baudon2017,Boudreau2004,Chia2012,Giambrone2011,Hartnell2003,Tuza2017}. Three of them deal with magic labelings~\cite{Boudreau2004,Giambrone2011,Hartnell2003}, one of them considers the game version of L(d,1)-labelings~\cite{Chia2012}, another one considers the game version of neighbor-sum-distinguishing edge labelings~\cite{Baudon2017} and, in a recent survey, Z.~Tuza~\cite{Tuza2017} surveys the area and proposes new labeling games such as the graceful game studied in this work. While the number of articles published in labeling games has been scarce so far, in the related area of graph colorings, there is a track of research concerning the `game chromatic number' comprising more than fifty published papers (see Tuza and Zhu's survey~\cite{Tuza2015}.)

In this work, we investigate the graceful game on simple graphs, proposed by Tuza~\cite{Tuza2017} in order to contribute to the study of the Rosa's conjectures. Informally, the graceful game is a two-players game on a connected graph in which Alice and Bob take turns labeling the
vertices of a graph $G$ with distinct numbers from the set $\{0,1,\ldots,|E(G)|\}$. Alice's goal is to gracefully label the graph as Bob's goal is
to prevent it from happening. In this work, we study winning strategies for Alice and Bob in the following families of graphs: paths, cycles, complete bipartite graphs, complete graphs, caterpillars, hypercubes, helms, webs, gear graphs, prisms and some powers of paths. This paper is organized as follows. Section~\ref{sec:Graceful} presents some definitions and auxiliary results used in our proofs. Section~\ref{sec:results} presents our main results on the graceful game for some classic families of graphs. In Section~\ref{sec:conclusion}, we present our conclusions.

\section{The Graceful Game}\label{sec:Graceful}

The \emph{Graceful Game} is defined in the following way: Alice and Bob alternately assign an unused label $f(v)\in \{0,\ldots,m\}$ to a previously unlabeled vertex $v$ of a given simple graph $G = (V(G),E(G))$ with $m$ edges. We call a vertex of $G$ \emph{free} if it is not labeled yet. If both endpoints of an edge $uv \in E$ are already labeled, then the \emph{label} of the edge $uv$ is defined as $|f(u)-f(v)|$. A move (labeling) is said to be \emph{legal} if, after it, all edge labels are distinct. In the Graceful Game, Alice \emph{wins} if the whole graph $G$ is gracefully labeled, and Bob \emph{wins} if he can prevent this.

It is well known that not every graph is graceful; in fact, most graphs are not graceful~\cite{Golomb1972}. For non-graceful graphs, it is immediate that Bob is the winner and, therefore, the game is completely determined for such graphs. In this work, we investigate classes of graphs for which it is possible to obtain a graceful labeling. 

The next two lemmas show properties of graceful games that are used throughout this work. Before presenting the lemmas, an additional definition is needed. Given a graph $G$ with $m$ edges and with a graceful labeling $f$, the \emph{complementary labeling} of $f$ is the labeling $\overline{f}$ defined as $\overline{f}(v) = m - f(v)$ for all $v \in V(G)$. The complementary labeling of a graceful labeling of $G$ is also graceful.

\begin{lemma}\label{lemma:firstlab0} 
Let $G$ be a simple graph with $m$ edges. Alice can only use the label $0$ (resp.~$m$) to label a vertex $v \in V(G)$ if $v$ is adjacent to every remaining free vertex or $v$ is adjacent to a vertex already labeled by Bob with $m$ (resp.~$0$).
\end{lemma}

\begin{proof} 
A graceful graph $G$ must have an edge with induced label $m$ and the only way to obtain it is by assigning labels $0$ and $m$ to two adjacent vertices. Thus, suppose Alice labels a vertex $v \in V(G)$ with $0$, without Bob having already labeled any vertex with $m$ and there is a free vertex not adjacent to $v$ in the graph. On Bob's next move, he assigns label $m$ to the free vertex that is not adjacent to $v$, making it impossible for Alice to gracefully label the graph. The case with $0$ and $m$ exchanged is analogous by the complementary labeling.
\end{proof}

Lemma~\ref{lemma:forceA} establishes properties where Alice is forced to use the label $0$ or the label $m$.

\begin{lemma}\label{lemma:forceA} 
Let $G$ be a simple graph with $m$ edges. If Bob assigns label $0$ (resp.~$m$) to a vertex $v \in V(G)$, such that $v$ has only one free neighbor or there are two free vertices in $G$ not adjacent to $v$, then Alice is forced to label a vertex adjacent to $v$ with $m$ (resp.~$0$). 
\end{lemma}

\begin{proof}  
First, suppose Bob labels $v \in V(G)$ with 0 (resp. $m$) and $v$ has exactly one free neighbor. If Alice does not assign $m$ (resp. 0) to the unique $v$'s neighbor, then Bob can label it with a label $j\neq m$, $j \in \{1,2,\ldots, m-1\}$. Similarly for the case where there are two free vertices in $G$ not adjacent to $v$, if Alice chooses to use $j\neq m$ on any vertex, then Bob assigns $m$ (resp. 0) to a vertex not adjacent to $v$. In both cases, Bob wins the game.
\end{proof}

\section{Main results}\label{sec:results}

In this section, we present our main results. A \textit{path graph} $P_n$ is a connected graph on $n$ vertices whose vertices can be arranged in a linear sequence ($v_0,v_1,\ldots,v_{n-1}$) in such a way that two vertices are adjacent if and only if they are consecutive in the linear sequence. Rosa~\cite{Rosa1967} proved that all paths are graceful. In Theorem~\ref{thm:labpath}, the graceful game is characterized for all paths.

\begin{theorem}\label{thm:labpath} 
Bob has a winning strategy for any $P_n$, $n\geq 4$. For $n=3$ the winner is the player who starts the game, and Alice has a winning strategy for $n \in \{1,2\}$.
\end{theorem}

\begin{proof} 
For the cases $P_1$ and $P_2$, Alice always wins given the fact there is only one way of labeling $P_1$ and, by $P_2$'s symmetry and by complementary labelling, there is only one way of gracefully labeling $P_2$. For $P_3$, if Bob starts, then he labels $v_1$ with $1$ and, no matter how Alice decides to label her next vertex, she never gets the edge label $2$, therefore losing the game. In contrast, if Alice starts, she labels $v_0$ with $1$. Now, independently of Bob's choice, the graph is graceful.

Next, consider $P_n$ with $n\geq 4$. When Bob is the first player, his strategy is to assign $0$  (resp. m) to $v_0$. By Lemma~\ref{lemma:forceA}, Alice's only option is to label $v_1$ with $m$. Then, Bob must label $v_2$ with any label but $1$. Now, regardless of the next moves, this graph can never be graceful since there are no possibilities of getting edge label $m-1$.

When Alice is the first player, by Lemma~\ref{lemma:firstlab0}, Alice must choose a label $i \in \{1, \ldots, m-1\}$. For $P_4$, we refer to Figure~\ref{pic:gracefulP4} for its only two graceful labelings. Wlog, if Alice labels $v_0$ (resp. $v_1$) with $i$, then Bob assigns $0$ to $v_1$ (resp. $v_0$) to win.
For $P_n$, $n\geq 5$, if she assigns $i$ to $v_j$ then Bob labels $v\in N(v_j)$ with $0$. If there does not exists $u\in N(v)$ then she loses the game. If there exists $u\in N(v)$ then Alice is forced to label $u$ with $m$.
If $i\not =1$ and $N(u) \not = \emptyset$, Bob assigns $1$ to a vertex $x \not \in N(u)$, then preventing the edge label $m-1$. Otherwise, Alice also loses the game. 
\end{proof}

\begin{figure}[!htb]
\centering
\begin{tikzpicture}[main_node/.style={circle,fill=white,draw,minimum size=.3em,inner sep=2pt]},scale=0.8]
    \node[main_node] (1) at (0,0) {\scriptsize $2$};
    \node[main_node] (2) at (1.5,0) {\scriptsize $1$};
    \node[main_node] (3) at (3,0) {\scriptsize $3$};
    \node[main_node] (4) at (4.5,0) {\scriptsize $0$};
    \draw (1) -- (2) -- (3) -- (4);
\end{tikzpicture}
\hspace{1in}
\begin{tikzpicture}[main_node/.style={circle,fill=white,draw,minimum size=.3em,inner sep=2pt]},scale=0.8]
    \node[main_node] (1) at (0,0) {\scriptsize $1$};
    \node[main_node] (2) at (1.5,0) {\scriptsize $2$};
    \node[main_node] (3) at (3,0) {\scriptsize $0$};
    \node[main_node] (4) at (4.5,0) {\scriptsize $3$};
    \draw (1) -- (2) -- (3) -- (4);
\end{tikzpicture}
\caption{Graceful labelings of $P_4$.}
\label{pic:gracefulP4}
\end{figure}
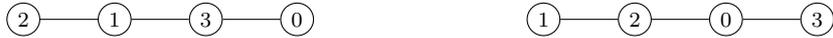

A \textit{complete graph} $K_n$ is a simple graph in which every pair of distinct vertices is connected by one edge. Golomb~\cite{Golomb1972} proved that a complete graph $K_n$ is graceful if and only if $n\leq 4$. The cases where $n = 1$ or $n = 2$ are trivial and resemble that of $P_1$ and $P_2$.

\begin{theorem}\label{thm:completegrp} 
Alice wins on $K_3$ and Bob on $K_4$, no matter who starts.
\end{theorem} 

\begin{proof} First, consider the complete graph $K_3$. Note that, in order to $K_3$ be graceful, two of its vertices must be assigned labels $0$ and $3$. So, when Alice is the first player, she can start by labeling the first vertex with label $0$. Now, no matter the vertex label Bob chooses, it is always possible to graceful label the graph with vertex labels $\{0, 1, 3\}$ or $\{0, 2, 3\}$. When Bob is the first player, his strategy is to label the graph using the set of vertex labels $\{0, 1, 2\}$ or $\{1, 2, 3\}$. Thus, he cannot label the first vertex with $0$ or $3$ because if he does, Alice's next move would be to label the second vertex with $3$, and the graph would turn out to be graceful. Based on the previous observations, we can assume that Bob labels the first vertex with $1$ or $2$. Consider that Bob starts assigning label $1$. Now, Alice knows she has to label the second vertex with $0$ or $3$. Assume that Alice labels the second vertex with $0$, creating the edge labeled $1$ (the case where Alice assigns label $3$ to the second vertex is complementary). By the rules of the Graceful game, to label the last vertex, Bob cannot use label $2$ since a repeated edge labeled $1=|2-1|$ would be induced, what constitutes an illegal move. This way, his only option is to label the last vertex with label $3$, therefore losing the game. 
The case where Bob chooses to label the first vertex with $2$ is analogous, by complementary labeling.

It is known~\cite{Golomb1972} and it can be verified by inspection that $K_4$ has only two graceful labelings (see Figure~\ref{fig:labK4}). From this fact, it can be deducted that no graceful labeling of $K_4$ assigns label $3$ to its vertices. Hence, when Alice is the first player, no matter what label she chooses for the first vertex, Bob can assign label $3$ for the second vertex. For the case when Bob starts the game, he can use $3$ to label the first vertex. In both cases, Bob wins the game.
\end{proof}

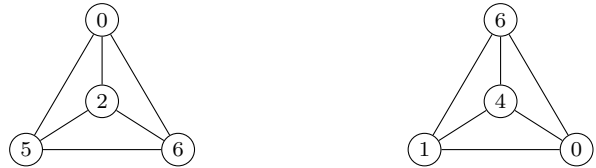
\begin{figure}[!htb]
\centering
\begin{tikzpicture}[main_node/.style={circle,fill=white,draw,minimum size=.3em,inner sep=2pt]},scale=0.4]
    \node[main_node,label=below:{}] (1) at (0,0) {\scriptsize $5$};
    \node[main_node,label=below:{}] (2) at (5,0) {\scriptsize $6$};
    \node[main_node,label=above:{}] (3) at (2.5,4.3) {\scriptsize $0$};
    \node[main_node,label=below:{}] (4) at (2.5,1.6) {\scriptsize $2$};
    \draw (1) -- (2) -- (3) -- (1);
    \draw (1) -- (4);
    \draw (2) -- (4);
    \draw (3) -- (4);
\end{tikzpicture}
\hspace{1in}
\begin{tikzpicture}[main_node/.style={circle,fill=white,draw,minimum size=.3em,inner sep=2pt]},scale=0.4]
    \node[main_node,label=below:{}] (1) at (0,0) {\scriptsize $1$};
    \node[main_node,label=below:{}] (2) at (5,0) {\scriptsize $0$};
    \node[main_node,label=above:{}] (3) at (2.5,4.3) {\scriptsize $6$};
    \node[main_node,label=below:{}] (4) at (2.5,1.6) {\scriptsize $4$};
    \draw (1) -- (2) -- (3) -- (1);
    \draw (1) -- (4);
    \draw (2) -- (4);
    \draw (3) -- (4);
\end{tikzpicture}
\caption{Two graceful labelings of $K_4$. One is the complementary labeling of the other.}
\label{fig:labK4}
\end{figure}

The next class of graphs considered in this work are the cycles. A \emph{cycle graph} $C_n$, with $n\geq 3$ vertices, is a connected simple graph such that all of its vertices can be arranged in a cyclic sequence $(v_0,v_1,\ldots,v_{n-1})$ such that two vertices are adjacent if and only if they are consecutive in the sequence. Rosa~\cite{Rosa1967} proved that the cycle graph $C_n$ is graceful if  and only if $n \equiv 0,3\pmod{4}$. Therefore, it is immediate that Bob is the winner when $n \not\equiv 0,3\pmod{4}$.

\begin{theorem}\label{thm:cyclegrp} 
Bob has a winning strategy for $C_n$, $n \geq 4$, and Alice wins on $C_3$.
\end{theorem}

\begin{proof} 
Since $C_3 \cong K_3$, the result for $C_3$ follows from Theorem~\ref{thm:completegrp}. Next, consider $C_4$. According to Lemma~\ref{lemma:firstlab0}, when Alice is the first player, her only options are to start labeling the vertices with $1$, $2$ or $3$. Moreover, it can be verified by inspection that $C_4$ has only two distinct graceful labelings (see Figure~\ref{fig:labC4}), which are complementary to each other.

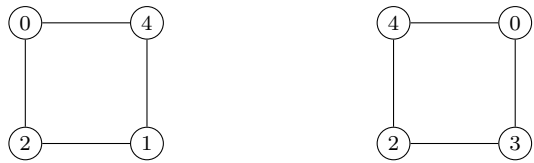
\begin{figure}[!htb]
\centering
\begin{tikzpicture}[main_node/.style={circle,fill=white,draw,minimum size=.3em,inner sep=2pt]},scale=0.8]
    \node[main_node,label=below:{}] (1) at (0,0) {\scriptsize $2$};
    \node[main_node,label=below:{}] (2) at (2,0) {\scriptsize $1$};
    \node[main_node,label=above:{}] (3) at (2,2) {\scriptsize $4$};
    \node[main_node,label=above:{}] (4) at (0,2) {\scriptsize $0$};
    \draw (1) -- (2) -- (3) -- (4) -- (1);
\end{tikzpicture}
\hspace{1in}
\begin{tikzpicture}[main_node/.style={circle,fill=white,draw,minimum size=.3em,inner sep=2pt]},scale=0.8]
    \node[main_node,label=below:{}] (1) at (0,0) {\scriptsize $2$};
    \node[main_node,label=below:{}] (2) at (2,0) {\scriptsize $3$};
    \node[main_node,label=above:{}] (3) at (2,2) {\scriptsize $0$};
    \node[main_node,label=above:{}] (4) at (0,2) {\scriptsize $4$};
    \draw (1) -- (2) -- (3) -- (4) -- (1);
\end{tikzpicture}
\caption{The graceful labelings of $C_4$. One is the complementary labeling of the other.}
\label{fig:labC4}
\end{figure}

Thus, no graceful labeling of $C_4$ has a vertex labeled $1$ and another labeled $3$. Hence, if Alice starts with $1$ or $3$, it is easy for Bob to win the game since it suffices to label a vertex with $3$ or $1$, respectively. It is also not harder for him to guarantee his victory in case Alice chooses to label a vertex $v$ with $2$ on the first move. Indeed, Bob labels a vertex $u \not \in N(v)$ with $1$ or $3$, making it impossible to obtain a graceful labeling for $C_4$.

Generalizing to $C_n$, $n>4$ and $n\equiv0,3$ (mod 4),  Bob again can exhaust Alice's chances of getting the edge label $n-1$. Remember that the only way to get the edge label $n-1$ is to have two adjacent vertices labeled with $n$ and $1$ or $0$ and $n-1$. Starting by the case where Bob is the first player, his strategy is to label a generic vertex $v_j \in V(C_n)$ with label $0$, then Alice's only option on her next move is to label $v_{j-1}$ or $v_{j+1}$ with $n$, getting the edge label $n$. Without loss of generality, suppose she chooses to label $v_{j+1}$. Now, Bob labels $v_{j+2}$ with $n-1$ or $v_{j-1}$ with $1$. In both cases he gets the edge label $1$ and makes it impossible for Alice to get the edge label $n-1$.

Now, if Alice starts the game, suppose she labels an arbitrary vertex $v_j \in V(C_n)$ with an arbitrary label $i \in \{1,\ldots,n-1\}$. If $i\not = 1$ then Bob labels $v_{j+3}$ with $0$, forcing Alice to label  $v_{j+4}$ (resp. $v_{j+2}$) with label $n$ on her next move. Thus, Bob labels $v_{j+2}$ (resp. $v_{j+4}$) with $1$ and wins the game. If $i = 1$ then Bob labels $v_{j+1}$ with $0$, forcing Alice to label $v_{j+2}$  with label $n$ on her next move. Bob wins the game again. 
\end{proof}

Next, we analyze the graceful game for complete bipartite graphs. A \emph{bipartite graph} is a graph $G=(V(G),E(G))$ such that there exists a partition $P=(X,Y)$ of $V(G)$ such that every edge of $E(G)$ connects a vertex in $X$ to a vertex in $Y$. A \emph{complete bipartite graph} $K_{p,q}$ is a simple bipartite graph in which each vertex of $X$ is joined to every vertex of $Y$, with $p = |X|$ and $q = |Y|$. Note that $|V(K_{p,q})|=p+q$ and $|E(K_{p,q})|=pq$.

\begin{theorem}\label{thm:compbipartide} 
Bob has a winning strategy for all $K_{p,q}$, $p,q\geq 2$. Alice wins the Graceful game in any star $K_{1,q}$ if she is the first player.
\end{theorem}

\begin{proof} 
The result for graphs $K_{1,0}$, $K_{1,1}$ and $K_{1,2}$ follows from Theorem~\ref{thm:labpath}. Thus, consider $K_{1,q}$, with $q \geq 3$ and let $(X,Y)$ be a bipartition of $K_{1,q}$. Without loss of generality, let $|X|=1$ and $|Y|=q$. When Alice is the first player, her strategy is to label the vertex in $X$ with $0$ or $q$. This way, regardless of what label or vertex Bob chooses for his later moves, the graph is graceful. On the other hand, if Bob is the first player, he labels the vertex in $X$ with any label other  than $0$ or $q$. Doing so, he precludes this graph from being gracefully labeled since it is impossible to have an edge with label $q$. 

Next, consider $K_{p,q}$, with $p,q\geq 2$, and let $(X,Y)$ be a bipartition of $K_{p,q}$. First, suppose that Bob is the first player. Without loss of generality, suppose Bob labels a vertex in $X$ with 0. Now, Alice is forced to label a vertex in $Y$ with $m=pq$, creating the edge label $m$. On the next move, Bob labels a vertex in $Y$ with 1. Note that the only option left for Alice to create the edge label $m-1$ is by assigning the label $m-1$ to a vertex in Y, so she is forced to make this move. Now, Bob labels another vertex in $Y$ with 2 and, by same reasoning, Alice is forced to label a vertex in $Y$ with $m-2$. This pattern goes on until the vertices in $Y$ are exhausted.

If $|Y|=q$ is even, Bob is the last player to label a vertex in $Y$ and he uses $z=\frac{q}{2}$. In this case, Alice can no longer create the edge label $m-z$, therefore loosing the game. If $|Y|=q$ is odd, then Alice labels the last vertex in $Y$ and she uses the label $m-z$, with $z = \left\lfloor q/2 \right\rfloor$. The next edge label she must guarantee the existence is $m-(z+1)$. The possible ways to create this edge label are: $|(m-(z+1))-0|$; $|(m-z)-1|$; $|(m-(z-1))-2|$; \ldots; $|(m-2)-(z-1)|$; $|(m-1)-z|$; and $|m-(z+1)|$. Since 1, 2, \ldots, $z-1$, $z$ and $m-z$, $m-(z-1)$, \ldots, $m-1$, $m$ are assigned to vertices in $Y$ and, there are no vertices left in $Y$ to label, the only way  Alice can create the edge label $m-(z+1)$ is by labeling a vertex in $X$ with $z+1$. However, this move creates a second edge label $z=|(z+1)-1|$. Therefore, Alice looses the game.

When Alice is the first player, she labels an arbitrary vertex $v$ with a label $i \in \{1,\ldots,m-1\}$. Without loss of generality, suppose $v\in X$. First, consider the case where $|X|=p=2$. The subcase $|X|=|Y|=2$ follows from Theorem~\ref{thm:cyclegrp} since $K_{2,2}\cong C_4$. 
Thus, we may assume that $|Y|\geq 3$. If $i\leq \frac{m}{2}$ (resp. $i\geq \frac{m}{2}+1$), then Bob must label a vertex in $Y$ with 0 (resp. $m$), forcing Alice to label the last available vertex in $X$ with $m$ (resp. $0$). Note that when $i=1$ (resp. $i=m-1$), Alice cannot create the edge label $m-1$, loosing the game. For all other values Bob labels a vertex in $Y$ with $2i-1$ (resp. $2i-(m-1)$). Now, Alice cannot label a vertex with 1 (resp. $m-1$) because it would create a repeated edge label $i-1$ (resp.~$m-i-1$). Since this is the only way she can create the edge label $m-1$, she looses.

We divide the proof for $K_{p,q}$ with $p\geq 3$ and $q\geq 2$ into two cases,  depending on the value of label $i$.

\textbf{Case 1:} $x < i < m-x$, where $x = \left(\frac{p-2}{2}\right)$ if $p=|X|$ is even; or $x = \left(\frac{p-1}{2}\right)$, otherwise.

Bob assigns label $0$ to a vertex in $Y$, thus generating the edge label $i$. Now, Alice is forced to label a vertex in $X$ with $m$, creating the edge label $m$. From this moment on, takes place the same pattern that occurs when Bob is the first player, until the vertices in $X$ are exhausted.

If $|X|$ is odd, Bob is the last player to label a vertex in $X$ and he uses label $x$. In this case, Alice can no longer create the edge label $m-x$, therefore loosing the game. If $|X|$ is even, then Alice labels the last vertex in $Y$ and she uses the label $m-x$. The next edge label she must guarantee the existence is $m-(x+1)$. However, since $1,2,\ldots,x$ and $m-x, m-(x-1),\ldots,m$ are assigned to vertices in $X$ and there are no vertices left in $X$ to label, the only way Alice can create the edge label $m-(x+1)$ is by labeling a vertex in $Y$ with $x+1$ (if $x+1$ was not already used by Alice in her first move). However, this move creates a second edge label $x=|(x+1)-1|$. Therefore, in any case, Alice looses the game.

\smallskip

\textbf{Case 2:} $i = k$ or $i = m-k$, for $1\leq k \leq x$, where $x = \left(\frac{p-2}{2}\right)$ if $p=|X|$ is even; or $x = \left(\frac{p-1}{2}\right)$, otherwise.

We show only the case $i=k$ since the case $i=m-k$ is analogous by complementary labeling. Thus, suppose $i=k$, as defined in the hypothesis. Bob labels a vertex in $Y$ with 0, creating the edge label $k$. Alice then labels a vertex in $X$ with $m$. Then, Bob labels another vertex in $X$ with 1, and so on, until the game reaches the point where Alice assigns label $m-(k-1)$ to a free vertex in $X$. Bob's next move would be to label a vertex in $X$ with $k$, but this was Alice's first move in the game. Now, in order to prevent Alice from labeling a vertex in $X$ with $m-k$, Bob must assign $m-2k$ to a vertex in $Y$. This way, if Alice tries to label a vertex in $X$ with $m-k$, she creates a second edge label $k$. Therefore, Bob wins the game since Alice cannot create the edge label $m-k$.
\end{proof}

An \emph{$\alpha$-labeling} of a graph $G$ on $m$ edges is a graceful labeling $g$ with the additional property that, for every edge $uv \in E(G)$, either $g(u)\leq k<g(v)$ or $g(v)\leq k < g(u)$, for some integer $k \in \{0,\ldots, m\}$. 
A \textit{caterpillar} $cat(k_1, k_2,\ldots, k_s)$ is a special tree obtained from a path $P = (v_1, v_2, \ldots,v_s)$, called \emph{spine}, by joining $k_j$ leaf vertices to $v_j$, for each $j \in \{1, \ldots, s\}$. Rosa~\cite{Rosa1967} proved that every caterpillar has an $\alpha$-labeling.

If a caterpillar $H$ with $m$ edges has diameter at most two, then $H$ is isomorphic to a star $K_{1,m}$ and, by Theorem~\ref{thm:compbipartide}, Alice wins the game on $H$ if she is the first player. On the other hand, Theorem~\ref{thm:catgrp} states that Bob has a winning strategy for all caterpillars with diameter at least three. In order to prove this result, the following lemma is needed.

\begin{lemma}\label{lemma:lableafcat} Let $H=cat(k_1, k_2,\ldots, k_s)$ be a caterpillar with $m$ edges and $v_j$ be an arbitrary vertex in the spine with $k_j>0$ adjacent leaves, for $j\in \{1,\ldots,s\}$. If $(i)$ there exists a leaf $u$ not yet labeled adjacent to $v_j$  and; $(ii)$ the colors $0$ or $m$ have not been used, then Alice cannot label $v_j$ with any color.
\end{lemma}

\begin{proof} 
Consider $H=cat(k_1, k_2,\ldots, k_s)$ be a caterpillar and $v_j$ be a vertex in the spine with $k_j>0$ adjacent leaves, for $j\in \{1,\ldots,s\}$. Suppose $v_j$ has at least one leaf neighbor not yet labeled, say $u$, and suppose that the colors $0$ or $m$ have not been assigned to any vertex so far. Consider Alice chooses to label $v_j$ with an arbitrary color $i$, $i\in \{1,\ldots,m-1\}$ (Lemma~\ref{lemma:firstlab0}). Given these conditions, Bob labels $u$ with $0$ or $m$. Since there is no other vertex adjacent to $u$ other than $v_j$, it is impossible to obtain the edge label $m$.
\end{proof}

\begin{theorem}\label{thm:catgrp} Bob has a winning strategy for all caterpillars with diameter at least three.
\end{theorem}

\begin{proof} 
Let $H=cat(k_1, k_2,\ldots, k_s)$ be a caterpillar with $m$ edges and with diameter at least three. For the case where $H$ is isomorphic to a path graph, the result follows from Theorem~\ref{thm:labpath}. Thus, suppose $H$ is not a path graph. When Bob is the first player, he assigns label $0$ to a leaf $u$ neighbor of a vertex $v_j$, for $j\in \{1,\ldots, s\}$. Since $u$ is only adjacent to $v_j$, Alice's only option is to label $v_j$ with $m$. Now, Bob wins the game by labeling a vertex not adjacent to $v_j$ with label $1$, making sure Alice cannot get edge label $m-1$.

Next, consider that Alice is the first player. By Lemma~\ref{lemma:lableafcat}, she can only label a leaf or a vertex $v_j$ in the spine whose $k_j=0$, $j\in \{1,\ldots, s\}$. Suppose that Alice chooses to label a leaf adjacent to a vertex $v_p$, $1 \leq p \leq s$, with an arbitrary color $i$, $i\in \{1,\ldots,m-1\}$ (by Lemma~\ref{lemma:firstlab0}). Then, Bob must assign $0$ to a leaf adjacent to a vertex $v_q$, $q=1,\ldots, s$ and $q\neq p$, forcing Alice to label $v_q$ with $m$. If $i=1$ then there are no possibilities left for Alice to label an edge $m-1$. If $i\neq 1$, then it is sufficient for Bob to label a vertex not adjacent to $v_q$ with label $1$. Consequently, there are no possibilities to obtain an edge label $m-1$. However, there is a problem regarding this strategy when it is addressed on a caterpillar $H=cat(1, k_2)$, with $k_2\geq 2$, and Alice starts by labeling the leaf adjacent to $v_1$ with a color $i\neq 1$. It is known that Bob's next move is to label a leaf adjacent to $v_2$ with $0$, forcing Alice to label $v_2$ with $m$ (Lemma~\ref{lemma:forceA}). Now, all free vertices available are adjacent to $v_2$. Thus, there is no vertex that labeled with $1$ guarantees Bob's victory. Therefore, we approach this case differently depending on the  parity of $m$.

First, suppose that $m$ is even. In this case, Bob labels $v_1$ with $1$ on the second move of the match, thus creating the edge label $i-1$. Notice that $i-1=m-m+(i-1)=m-(m-(i-1))$. Since the only vertices left to label are leaves adjacent to $v_2$, from this moment on, any edge label that is created shall have the value $m$ as one of its ends. Notice that, for any $i\in\{2,\ldots,m-1\}$, there exists a value $x \in \{2,\ldots,m-1\}\backslash \{i\}$ that would necessarily be used to label a leaf adjacent to $v_2$ and that is such that $m-x=i-1$. This value is $x=m-(i-1)$. Therefore, we conclude that a caterpillar $H=cat(1, k_2)$ has, in any case, two equal edge labels; so, it cannot be graceful.

Now, suppose that $m$ is odd. In this case, the number of integers from $2$ to $m-1$ is not even. Thus, we cannot apply the same strategy addressed in the case where $m$ is even, since there would be no label $x$ satisfying both conditions $m-x=i-1$ and $x\neq i$ in case $i=\frac{m+1}{2}$. Therefore, on the second play Bob must label $v_1$ with $0$, creating the edge label $i$. By Lemma~\ref{lemma:forceA}, Alice is forced to label $v_2$ with $m$. First, notice that $i=m-m+i=m-(m-i)$. Now, as the only remaining free vertices are the leaves adjacent to $v_2$, any edge label that is obtained from this moment on has label $m$ assigned to one of its ends. We claim that, for any $i\in \{1,\ldots,m-1\}$, there exists a value $y \in \{1,\ldots,m-1\}\backslash \{i\}$ that would necessarily be used to label a leaf adjacent to $v_2$ and which is such that $m-y=i$. Note that this value is $y=m-i$. By same reasoning as the previous case, this graph cannot be graceful. This concludes the case where Alice labels a leaf of the graph on the first move. 

Now, suppose that Alice labels with an arbitrary color $i\in \{1,\ldots,m-1\}$ (Lemma~\ref{lemma:firstlab0}) a vertex $v_j$, $j\in \{1,\ldots,s\}$, in the spine whose $k_j=0$. First, suppose that $H$ is not isomorphic to $cat(k_1,0,k_3)$, $k_1,k_3>0$. On the next move, Bob must assign label $0$ to a leaf adjacent to a vertex $v_p$, $p \in \{1,\ldots, s\}$, where $p\neq j-1$ and $p\neq j+1$. By Lemma~\ref{lemma:forceA}, Alice is forced to label $v_p$ with $m$. If $i\neq 1$, then Bob labels a vertex not adjacent to $v_p$ with $1$ (Note that the existence of a vertex not adjacent to $v_p$ is guaranteed given the fact that the smaller graph we may have with a vertex $v_j$, $j\in \{1,\ldots,s\}$, in the spine whose $k_j=0$ is $cat(k_1, 0, k_3)$, where $k_1, k_3>0$.) This guarantees Alice cannot label an edge $m-1$. If $i=1$, then it is already guaranteed.

However, we cannot apply the previous strategy when $H=cat(k_1, 0, k_3)$, with $k_1,k_3>0$. For this case, Alice starts by assigning label $i\in \{1,\ldots,m-1\}$ to $v_2$. On the next move, Bob assigns $0$ to a leaf adjacent to $v_1$ (for $v_3$ is analogous). Alice is now forced to label $v_1$ with $m$ (Lemma~\ref{lemma:forceA}), creating the edge labels $m$ and $m-i$. If $i\neq 1$, then Bob assigns label $1$ to any vertex not adjacent to $v_1$ and wins the game. On the other hand, if $i = 1$, then Bob labels $v_3$ with $2$, getting the edge label $1=|2-1|$. From this moment on, neither one of the players can label a leaf of $v_3$ with $3$ or, a leaf of $v_1$ with $m-1$, because they would be labeling another edge label $1$. Since all vertex labels $0,1,\ldots,m$ must be used, at some point, a leaf adjacent to $v_1$ will be assigned label $3$, creating an edge label $m-3$, and a leaf adjacent to $v_3$ will be assigned label $m-1$, creating a repeated edge label $m-3$. Therefore, it is not possible to complete a graceful labeling from this configuration.
\end{proof}

The next class of graphs considered in this work are the wheel graphs. A \emph{wheel} $W_n$ is a graph formed by connecting a single vertex $v_n$ to all vertices $v_0,v_1,\ldots,v_{n-1}$ of a cycle $C_n$, where $n\geq 3$. By the definition, a wheel $W_n$ has $n+1$ vertices and $2n$ edges. In this work, the vertex $v_n$, which is adjacent to all the other vertices of $W_n$, is called the \emph{central vertex} of $W_n$.
Figure~\ref{fig:wheels} illustrate some wheel graphs with graceful labelings.

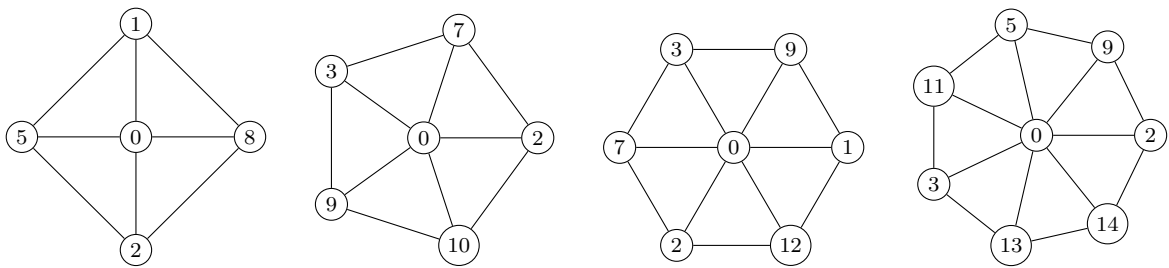
\begin{figure}[!htb]
\centering

\begin{tikzpicture}[scale=1.5]
\tikzstyle{every node}=[draw,shape=circle,minimum size=.3em,inner sep=1.8pt];
\path (0:0cm) node (v0) {\scriptsize $0$};
\path (0:1cm) node (v1) {\scriptsize $8$};
\path (90:1cm) node (v2) {\scriptsize $1$};
\path (2*90:1cm) node (v3) {\scriptsize $5$};
\path (3*90:1cm) node (v4) {\scriptsize $2$};
\draw (v0) -- (v1)
(v0) -- (v2)
(v0) -- (v3)
(v0) -- (v4)
(v1) -- (v2) -- (v3) -- (v4) -- (v1);
\end{tikzpicture}
\hspace{1em}
\begin{tikzpicture}[scale=1.5]
\tikzstyle{every node}=[draw,shape=circle,minimum size=.3em,inner sep=1.9pt];
\path (0:0cm) node (v0) {\scriptsize $0$};
\path (0:1cm) node (v1) {\scriptsize $2$};
\path (72:1cm) node (v2) {\scriptsize $7$};
\path (2*72:1cm) node (v3) {\scriptsize $3$};
\path (3*72:1cm) node (v4) {\scriptsize $9$};
\path (4*72:1cm) node (v5) {\scriptsize $10$};
\draw (v0) -- (v1)
(v0) -- (v2)
(v0) -- (v3)
(v0) -- (v4)
(v0) -- (v5)
(v1) -- (v2) -- (v3) -- (v4) -- (v5) -- (v1);
\end{tikzpicture}
\hspace{1em}
\begin{tikzpicture}[scale=1.5]
\tikzstyle{every node}=[draw,shape=circle,minimum size=.3em,inner sep=1.9pt];
\path (0:0cm) node (v0) {\scriptsize $0$};
\path (0:1cm) node (v1) {\scriptsize $1$};
\path (60:1cm) node (v2) {\scriptsize $9$};
\path (2*60:1cm) node (v3) {\scriptsize $3$};
\path (3*60:1cm) node (v4) {\scriptsize $7$};
\path (4*60:1cm) node (v5) {\scriptsize $2$};
\path (5*60:1cm) node (v6) {\scriptsize $12$};
\draw (v0) -- (v1)
(v0) -- (v2)
(v0) -- (v3)
(v0) -- (v4)
(v0) -- (v5)
(v0) -- (v6)
(v1) -- (v2) -- (v3) -- (v4) -- (v5) -- (v6) -- (v1);
\end{tikzpicture}
\hspace{1em}
\begin{tikzpicture}[scale=1.5]
\tikzstyle{every node}=[draw,shape=circle,minimum size=.3em,inner sep=1.9pt];
\path (0:0cm) node (v0) {\scriptsize $0$};
\path (0:1cm) node (v1) {\scriptsize $2$};
\path (51.42:1cm) node (v2) {\scriptsize $9$};
\path (2*51.42:1cm) node (v3) {\scriptsize $5$};
\path (3*51.42:1cm) node (v4) {\scriptsize $11$};
\path (4*51.42:1cm) node (v5) {\scriptsize $3$};
\path (5*51.42:1cm) node (v6) {\scriptsize $13$};
\path (6*51.42:1cm) node (v7) {\scriptsize $14$};
\draw (v0) -- (v1)
(v0) -- (v2)
(v0) -- (v3)
(v0) -- (v4)
(v0) -- (v5)
(v0) -- (v6)
(v0) -- (v7)
(v1) -- (v2) -- (v3) -- (v4) -- (v5) -- (v6) -- (v7) -- (v1);
\end{tikzpicture}
\caption{Wheels with graceful labelings.}
\label{fig:wheels}
\end{figure}

In 1987, C.~Hoede and H.~Kuiper~\cite{Hoede1987} proved that all wheels are graceful and this result was later rediscovered by R.~Frucht~\cite{Frucht1979} in 1979. A useful fact about graceful labelings of wheels is that there exist no graceful labeling $f$ of $W_n$ that assigns the label $n$ to its central vertex. Note that if such a labeling existed, the labels $0$ and $2n$ should be assigned to two adjacent vertices $v_j$ and $v_{j+1}$ (indices modulo $n$), $1\leq j \leq n$, so as to induce the edge label $2n$. However, such an assignment would generate two repeated edge labels with value $n$, given by $|f(v_n)-f(v_{j+1})|=|n-2n|$ and $|f(v_n)-f(v_j)|=|n-0|$. This observation leads to the following result.

\begin{lemma}[Frucht~\cite{Frucht1979}]\label{lemma:labeln}
 Let $v_n$ be the central vertex of the wheel graph $W_n$. There exist no graceful labeling $f$ of $W_n$ with $f(v_n)=n$.\qed
\end{lemma}

Lemma~\ref{lemma:labeln} immediately implies the following result.

\begin{theorem}
 When Bob is the first player, he has a winning strategy for all wheel graphs.
\end{theorem}

\begin{proof}
 Consider a wheel graph $W_n$, with $n\geq 3$. Since Bob is the first player, he starts by assigning label $n$ to $W_n$'s central vertex and, by Lemma~\ref{lemma:labeln}, such an assignment precludes Alice from obtaining a graceful labeling of $W_n$.
\end{proof}

\begin{theorem}\label{thm:wheels}
 Bob has a winning strategy for the wheel graphs $W_3$, $W_4$ and $W_5$ even when Alice is the first player.
\end{theorem}

\begin{proof}
Consider $W_n$ with $3\leq n \leq 4$ and suppose that Alice is the first player. For $n=3$, the result follows from Theorem~\ref{thm:completegrp} since $W_3 \cong K_4$. Thus, consider $n=4$. Consider a graceful labeling $f$ of $W_4$ represented as a 5-tuple $(a_0,a_1,a_2,a_3,a_4)$ where $a_j = f(v_j)$, for $0\leq j\leq 4$. Considering this representation, all graceful labelings of $W_4$, obtained through computational search, are:
\begin{align*}
(2, 5, 1, 8, 0)&\quad (3, 7, 6, 8, 0)\quad (0, 8, 4, 6, 1)\quad (0, 8, 1, 5, 2)\\
(0, 8, 3, 7, 6)&\quad (0, 8, 2, 4, 7)\quad (7, 3, 6, 0, 8)\quad (5, 1, 2, 0, 8).
\end{align*}

Since Alice aims to win, she has two choices: either (i)  she labels a vertex other than $v_4$ with label $i = 4$ (by doing this, she precludes Bob from assigning label $4$ to the central vertex (see Lemma~\ref{lemma:labeln})); or (ii) she labels the central $v_4$ with a label $i\in \{0,1,\ldots,8\}$ and $i\neq 4$ (by Lemma~\ref{lemma:firstlab0} and Lemma~\ref{lemma:labeln}).
\textbf{Case (i):} Alice labels a vertex $v_j$ other than $v_4$ with label $i = 4$, say that she labels $v_0$. In this case, Bob then assigns label $0$ to the central vertex $v_4$, which forces Alice to assign label $8$ to a neighbor of $v_4$: the only valid choice at this point is to assign $8$ to vertex $v_2$ so as to avoid a repeated edge label $4$. Now, Bob assigns label $1$ to $v_1$ generating edge labels $1,3,4,7,8$ up to this point. The reader can verify by inspection that any remaining vertex label $2,3,5,6,7$ does not generate a graceful labeling of $W_4$ when assigned to the remaining free vertex $v_3$. Therefore, Bob wins the game.  \textbf{Case (ii):} Suppose, Alice labels the central $v_4$ with a label $i\in \{0,1,\ldots,8\}$ and $i\neq 4$. From the graceful labelings shown above, we deduct that $i\neq 3$ and $i\neq 5$. It is also not difficult to deduct from these labelings that, for every remaining valid value of $i$, there exists a value $\ell_i$ that Bob can assign to a free vertex of $W_4$ in order to preclude Alice from obtaining a graceful labeling of the graph; these values are: $\ell_0 = \ell_2 = \ell_6 = \ell_8 = 4$ and $\ell_1 = \ell_7 = 3$. Therefore, Bob is the winner, and the result follows.

\medskip

Next, consider $W_5$ and recall that $|E(W_5)| = 10$. Suppose that Alice is the first player. There are two main cases to consider depending on which vertex Alice chooses to label first. 

\textbf{Case 1:} Alice labels a vertex $v_j$ with degree three, for $0\leq j\leq 4$. In this case, if Alice labels $v_j$ with a label $i \in \{1,\ldots,9\}\backslash \{5\}$, then Bob can now assign label $5$ to central vertex $v_5$. By Lemma~\ref{lemma:labeln}, Bob wins the game since there is no graceful labeling of $W_5$ with label $5$ assigned to its central vertex. Therefore, we can now assume that $i=5$. Then, Bob assigns label $0$ to vertex $v_{j-1}$ or $v_{j+1}$ (indices taken modulo $n$). Without loss of generality suppose he chooses $v_{j-1}$. Thus, Alice is forced to assign label $10$ to a neighbor of $v_{j-1}$ ($v_5$ or $v_{j-2}$) so as to generate the edge label $10$. Alice cannot assign $10$ to $v_5$ since this would create a repeated edge label $5 = |f(v_5)-f(v_j)| = |f(v_{j-1})-f(v_j)|$. Hence, Alice chooses the remaining option that is to assign label $10$ to vertex $v_{j-2}$. Now, in order to preclude Alice from generate the edge label $9$, Bob assigns label $9$ to vertex $v_{j-3}$, generating the edge label $1 = |f(v_{j-3})-f(v_{j-2})|$. The only way for Alice to generate edge label $9$ consists in assigning $1$ to the central vertex $v_5$. However, this move is not allowed since it would generate the repeated edge label $1 = |f(v_5)-f(v_{j-1})|$. Therefore, Bob is the winner.

\textbf{Case 2:} Alice labels vertex $v_5$. Since Alice aims to win the game, she cannot assign label $5$ to $v_5$ (see Lemma~\ref{lemma:labeln}). 

\emph{Subcase 2.1:} Alice labels $v_5$ with a label $i$ such that $3\leq i \leq 8$ and $i\neq 5$. Then, Bob assigns label $0$ to a vertex $v_j$, adjacent to $v_5$, thus generating the edge label $i$. Such a move forces Alice to assign label $10$ to a neighbor of $v_j$ so as to generate the edge label $10$ (Alice's move also generates the edge label $10-i$). However, now it is Bob's turn and he assigns label $1$ to the remaining free neighbor of $v_j$ (it is not difficult to verify that the induced edge labels until now are $10$, $10-i$, $i$, $i-1$ and $1$, and that all of them are distinct). With such a move, Bob precludes Alice from completing a graceful labeling of the graph since it is not possible to generate the edge label $9$. 

\emph{Subcase 2.2:} Alice labels $v_5$ with label $i=2$. Now its is Bob's turn and he assigns label $0$ to a vertex $v_j$, adjacent to $v_5$, thus generating the edge label $2$. Such a move forces Alice to assign label $10$ to a neighbor of $v_j$, say $v_{j+1}$, so as to generate the edge label $10$ (Alice's move also generates the edge label $8 = |f(v_{j+1})-f(v_5)|$). However, now it is Bob's turn and he assigns label $1$ to vertex $v_{j-2}$ (it is not difficult to verify that the induced edge labels until now are $10$, $8$, $2$ and $1$, and that all of them are distinct). With such a move, Bob precludes Alice from completing a graceful labeling of the graph since it is not possible to generate the edge label $9$.

\emph{Subcase 2.3:} Alice labels $v_5$ with label $i\in \{1,9\}$. We only prove the case $i=1$ since the other case is analogous by complementary labeling. In this case, Bob assigns label $9$ to a vertex $v_j$, adjacent to $v_5$, thus generating the edge label $8=|f(v_j)-f(v_5)|$. Now, Alice assigns a label $\ell$, $2\leq \ell \leq 8$, to a free vertex. Since $n\geq 5$, vertex $v_j$ has at least one free neighbor, say $v_{j+1}$. Thus, in his next move, Bob assigns label $10$ to $v_{j+1}$ (generating edge labels $1$ and $9$), forcing Alice to assign label $0$ to the free vertex $v_{j+2}$ so as to generate the edge label $10$. However, this move is not allowed since it generates a repeated edge label $1 = |f(v_5)-f(v_{j+2})| = |f(v_j)-f(v_{j+1})|$. Therefore, Bob is the winner.

\emph{Subcase 2.4:} Alice labels $v_5$ with label $i=0$ (resp.~$i=10$). Consider a graceful labeling $f$ of $W_5$ represented as a $6$-tuple $(a_0,a_1,\ldots,a_5)$, where $a_j = f(v_j)$, for $0\leq j \leq 5$. Figure~\ref{fig:labelingW5} exhibits all graceful labelings of $W_5$ that assign label $0$ to its central vertex $v_5$. Note that none of these graceful labelings assign label $5$ to a vertex of $W_5$. Therefore, in this case, after Alice assigns label $0$ (resp.~$10$) to $v_5$, Bob then assign label $5$ to any vertex of $W_5$, thus winning the game, and the result follows.
\end{proof}

\begin{figure}[!htb]
\centering
\begin{tabu} to 0.9\textwidth {  X[l]  X[l]  X[l]  X[l]  }
 (10, 1, 7, 3, 8, 0) & (10, 2, 7, 3, 9, 0) & (10, 3, 1, 9, 4, 0) & (10, 7, 9, 1, 6, 0)
\end{tabu}
\caption{All graceful labelings of wheel graph $W_5$ that assign label $0$ to its central vertex $v_5$.}\label{fig:labelingW5}
\end{figure}

We note that the approach taken in order to prove Theorem~\ref{thm:wheels} for $W_5$ can almost be successfully extended for arbitrary wheels $W_n$ with $n\geq 6$, with exception of Subcase 2.4. However, we conjecture that Bob has a winning strategy for all these graphs.

\begin{conjecture}
 Bob has a winning strategy for all wheel graphs $W_n$ with $n\geq 6$.
\end{conjecture}

We now investigate a class of graphs related to wheel graphs. A \emph{gear graph} $G_n$, with $n\geq 3$ vertices, is a simple graph obtained by subdividing each edge of the outer $n$-cycle $(v_0,v_1,\ldots,v_{n-1})$ of a wheel graph $W_n$ exactly once. The vertices of $G_n$ are named as follows: $v_n\in V(G_n)$ is the original central vertex of $W_n$, $v_0,v_1,\ldots,v_{n-1}$ are the original vertices of the outer $n$-cycle (these vertices are adjacent to $v_n$) and, for each $j\in\{0,\ldots,n-1\}$, $w_j$ is the vertex adjacent to $v_j$ and $v_{j+1}$, indices taken modulo $n$. Note that, by the definition, $G_n$ has $2n+1$ vertices and $3n$ edges. Figure~\ref{fig:gear} exhibits graph $G_3$ with a graceful labeling.

\begin{figure}[!htb]
\centering
\begin{tikzpicture}[scale=1.5]
\tikzstyle{every node}=[draw,shape=circle,minimum size=.3em,inner sep=1.9pt];
\path (0:0cm) node (v0) {\scriptsize $4$};
\path (0:1cm) node (v1) {\scriptsize $0$};
\path (60:1cm) node (v2) {\scriptsize $6$};
\path (2*60:1cm) node (v3) {\scriptsize $5$};
\path (3*60:1cm) node (v4) {\scriptsize $8$};
\path (4*60:1cm) node (v5) {\scriptsize $1$};
\path (5*60:1cm) node (v6) {\scriptsize $9$};
\draw (v0) -- (v2)
(v0) -- (v4)
(v0) -- (v6)
(v1) -- (v2) -- (v3) -- (v4) -- (v5) -- (v6) -- (v1);
\end{tikzpicture}
\caption{Gear graph $G_3$ with a graceful labeling.}
\label{fig:gear}
\end{figure}
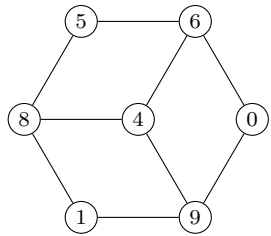

Ma and Feng~\cite{Ma1984} proved in 1984 that all gear graphs are graceful. In the next theorem, we show that Bob has a winning strategy for all gear graphs.

\begin{theorem}\label{thm:gears}
 Bob has a winning strategy for all gear graphs.
\end{theorem}

\begin{proof}
 Let $G_n$ be a gear graph with $m = 3n$ edges, $n\geq 3$. When Bob is the first player, he assigns label $0$ to a vertex $w_j$ of degree two, which forces Alice to assign label $m$ to a neighbor $v_j$ or $v_{j+1}$ of $w_j$. At the next round, Bob assigns label $1$ to the free neighbor of $w_j$, thus winning the game since the edge label $m-1$ cannot be obtained.

 Now, assume that Alice is the first player. First, suppose that $n \geq 4$. Alice starts by assigning a label $i\in \{1,\ldots,m-1\}$ to a vertex $u$ of $G_n$. There are three cases regarding the choice of $u$. 
 
 \textbf{Case 1:} $u = w_j$, for $0 \leq j \leq n-1$. Without loss of generality, Bob assigns label $0$ to vertex $w_{j+2}$ (indices taken modulo n), thus forcing Alice to assign label $m$ to a neighbor $v_{j+2}$ or $v_{j+3}$ of $w_{j+2}$. If $i\neq 1$, then Bob now assigns $1$ to the free neighbor of $w_{j+2}$, thus winning the game. On the other hand, if $i=1$, then Bob assigns label $m-2$ to the free neighbor of $w_j$, thus winning the game.
 
 \textbf{Case 2:} $u = v_j$, for $0 \leq j \leq n-1$. First, suppose that $i\neq m-1$. Then, Bob assigns label $0$ to $w_j$, which forces Alice to assign label $m$ to $v_{j+1}$. If $i=1$, Bob wins the game since there is no possibility for Alice to create the edge label $m-1$. On the other hand, if $2\leq i \leq m-2$, then Bob assigns label $1$ to vertex $w_{j-1}$, also winning the game.
 Now, suppose that $i=m-1$. In this case, Bob assigns label $m$ to $w_j$, which forces Alice to assign label $0$ to $v_{j+1}$. Hence, Bob wins the game since there is no possibility for Alice to create the edge label $m-1$.
 
 \textbf{Case 3:} $u = v_n$, the central vertex of $G_n$. First, suppose $i\neq m-1$. For the second move, Bob assigns label $0$ to a vertex $v_j$, which forces Alice to assign label $m$ to $w_{j-1}$ or $w_{j}$, say $w_j$. In his next move, Bob assigns label $1$ to $w_{j-1}$ if $i\neq 1$ or he assigns label $m-1$ to any $w_{p}$ with $p\neq j-1$ if $i=1$. In both cases, Bob wins the game since there is no possibility for Alice to create the edge label $m-1$. Now, suppose that $i = m-1$. In this case, Bob assigns label $m$ to a vertex $v_j$, forcing Alice to assign label $0$ to a free neighbor of $v_j$, say $w_{j-1}$. Then, now Bob assigns label $2$ to $w_j$ and wins the game.
 
 \smallskip
 
 Now, consider the gear graph $G_3$. As in the previous case, Alice starts by assigning a label $i\in \{1,\ldots,8\}$ to a vertex $u$ of $G_3$ that can be the central vertex $v_3$, a vertex $v_j$, or a vertex $w_j$, for $0\leq j \leq 2$. Consider any graceful labeling $f$ of $G_3$ represented as a 7-tuple $(f(v_3),f(v_0),f(w_0),f(v_1),f(w_1),f(v_2)$, $f(w_2))$. Figure~\ref{fig:firstlabelingsG3} shows all graceful labelings of $G_3$ that assign the edge label $m=9$ to an edge $v_jw_j$ (that is, that assign the vertex label $0$ to any vertex other than the central vertex). By inspection of these graceful labelings, we obtain that: (i) there is no graceful labeling $f$ of $G_3$ with  $f(v_3) \in\{1,2,3,4\}$ and the vertex label $0$ assigned to a neighbor of $v_3$; and (ii) there is no graceful labeling $f$ of $G_3$ with  $f(v_3) \in\{5,6,7,8\}$ and the vertex label $9$ assigned to a neighbor of $v_3$. 
 From these two facts, we obtain that Alice cannot start the game by labeling the central vertex $v_3$ since Bob wins by choosing between $0$ and $9$ to label a neighbor of $v_3$. Thus, we may assume that Alice starts by labeling a vertex other than $v_3$.

\begin{figure}[!htb]
\centering
\begin{tabu} to 0.9\textwidth { | X[l] | X[l] | X[l] | X[l] |}
 \hline
 (1, 9, 0, 3, 2, 8, 4) & (1, 9, 0, 4, 5, 7, 2) & (1, 9, 0, 4, 6, 7, 2) & (1, 9, 0, 5, 2, 8, 7)\\
 \hline
 (1, 9, 0, 6, 4, 5, 2) & (4, 9, 0, 6, 5, 8, 1) & (1, 9, 0, 6, 7, 5, 2) & (3, 9, 0, 7, 2, 4, 1)\\
 \hline
  (1, 9, 0, 7, 2, 4, 5) & (1, 9, 0, 7, 2, 4, 8) & (1, 9, 0, 7, 3, 6, 8) & (1, 9, 0, 7, 4, 6, 5)\\
 \hline 
 (3, 9, 0, 7, 5, 6, 1) & (3, 9, 0, 7, 8, 6, 1) & (3, 9, 0, 8, 1, 5, 6) & (2, 9, 0, 8, 3, 6, 7)\\
 \hline 
 (3, 9, 0, 8, 4, 1, 2) & (3, 9, 0, 8, 4, 5, 2) & (3, 9, 0, 8, 5, 1, 2) & (2, 9, 0, 8, 5, 6, 4)\\
 \hline
 (7, 0, 9, 1, 4, 3, 5) & (6, 0, 9, 1, 4, 8, 7) & (6, 0, 9, 1, 5, 4, 7) & (6, 0, 9, 1, 5, 8, 7)\\
 \hline
 (7, 0, 9, 1, 6, 3, 2) & (6, 0, 9, 1, 8, 4, 3) & (6, 0, 9, 2, 1, 3, 8) & (6, 0, 9, 2, 4, 3, 8)\\
 \hline
 (8, 0, 9, 2, 5, 3, 4) & (8, 0, 9, 2, 6, 3, 1) & (8, 0, 9, 2, 7, 5, 1) & (8, 0, 9, 2, 7, 5, 4)\\
 \hline 
 (6, 0, 9, 2, 7, 5, 8) & (8, 0, 9, 3, 2, 4, 7) & (5, 0, 9, 3, 4, 1, 8) & (8, 0, 9, 3, 5, 4, 7)\\
 \hline 
 (8, 0, 9, 4, 7, 1, 2) & (8, 0, 9, 5, 3, 2, 7) & (8, 0, 9, 5, 4, 2, 7) & (8, 0, 9, 6, 7, 1, 5)\\
 \hline
\end{tabu}
\caption{All graceful labelings of gear graph $G_3$ with the edge label $9$ assigned to an edge of the cycle on $6$ vertices induced by $V(G_3)\backslash \{v_3\}$.}\label{fig:firstlabelingsG3}
\end{figure}
 
 First, consider that Alice starts by labeling a vertex of degree three, say $v_1$, with label $i\in\{1,\ldots,8\}$. Then, Bob assigns label $0$ to $v_3$, which forces Alice to assign label $9$ to a free neighbor of $v_3$, say $v_0$. If $i\neq 1$ and $i\neq 8$, then it is sufficient Bob to assign label $1$ to the last free neighbor of $v_3$ in order to preclude Alice from generating the edge label $8$. On the other hand, if $i=1$ (resp.~$i=8$), then Bob assigns label $8$ (resp.~$1$) to $w_1$, therefore, winning the game.  
 
 Now, consider that Alice starts by labeling a vertex of degree two, say $w_1$, with label $i\in\{2,\ldots,7\}$. Then, Bob assigns label $0$ to any neighbor of $w_1$, say $v_2$, which forces Alice to assign label $9$ to a free neighbor of $v_2$ ($v_3$ or $w_2$). Note that Alice cannot assign label $9$ to $w_2$ since Bob can use the information provided by Figure~\ref{fig:firstlabelingsG3} in order to assign an appropriate label to $v_3$ in his next turn so as to preclude Alice from obtaining a graceful labeling of $G_3$. Thus, we have that Alice assigns label $9$ to $v_3$. Since $i\neq 8$ and $i\neq 1$, it suffices Bob to assign label $1$ to vertex $w_2$ of degree two in order to preclude Alice from generating the edge label $8$, thus winning the game. When $i = 1$, for his first move, Bob assigns label $0$ to $v_2$, then Alice assigns label $9$ to $v_3$ and, finally, Bob assigns label $8$ to $w_0$, thus winning the game. The last case is $i=8$. In this case, for his first move, Bob assigns label $0$ to a degree-two vertex not adjacent to $w_1$, say $w_0$. This forces Alice to assign label $9$ to a neighbor of $w_0$. Independently of Alice's choice, Bob can now assign label $5$, $6$ or $7$ to the central vertex $v_3$ in order to win the game (by Figure~\ref{fig:firstlabelingsG3}).
\end{proof}

Next, we consider another class of graphs related to wheel graphs. A \emph{helm} $H_n$, with $n\geq 3$, is the graph obtained from $W_n$ as follows: for each non-central vertex $v_j$ of $W_n$, create a new vertex and link $v_j$ with this new vertex. By the definition, a helm $H_n$ is a graph on $3n$ edges and $2n+1$ vertices, that are named as follows: the unique vertex with degree $n$ is called \emph{center} and is denoted by $v_0$; the neighbors of $v_0$ are called \emph{cycle vertices} and are denoted by $v_1,\ldots,v_n$; the remaining vertices are called \emph{pendent vertices} and are denoted by $v_{n+1},\ldots, v_{2n}$. Moreover, we consider that $v_k$ and $v_{n+k}$ are adjacent and are arranged linearly, $1\leq k \leq n$. In 1984, Ayel~\cite{Ayel1984} proved that all helms are graceful. Figure~\ref{fig:helm} exhibits $H_4$ with a graceful labeling. Theorem~\ref{thm:helms} characterizes the graceful game for helms and its proof uses the following lemma.

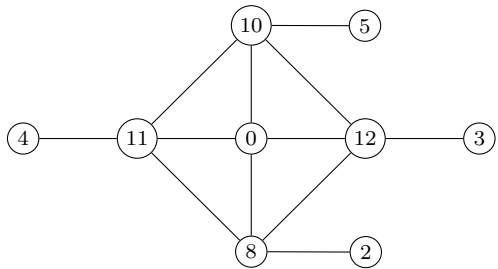
\begin{figure}[!htb]
\centering
\begin{tikzpicture}[scale=1.5]
\tikzstyle{every node}=[draw,shape=circle,minimum size=.3em,inner sep=1.8pt];
\path (0:0cm) node (v0) {\scriptsize $0$};
\path (0:1cm) node (v1) {\scriptsize $12$};
\path (90:1cm) node (v2) {\scriptsize $10$};
\path (2*90:1cm) node (v3) {\scriptsize $11$};
\path (3*90:1cm) node (v4) {\scriptsize $8$};

\path (0:2cm) node (v5) {\scriptsize $3$};
\path (2*90:2cm) node (v6) {\scriptsize $4$};
\path (7*45:1.42cm) node (v7) {\scriptsize $2$};
\path (45:1.41cm) node (v8) {\scriptsize $5$};
\draw (v0) -- (v1)
(v0) -- (v2)
(v0) -- (v3)
(v0) -- (v4)
(v1) -- (v5)
(v3) -- (v6)
(v4) -- (v7)
(v2) -- (v8)
(v1) -- (v2) -- (v3) -- (v4) -- (v1);
\end{tikzpicture}
\caption{Helm $H_4$ with a graceful labeling.}
\label{fig:helm}
\end{figure}

\begin{lemma}\label{lemma:Alabcyclevtx}
Given a \emph{helm} $H_n$, for $j \in\{1,\ldots,n\}$, Alice can label a cycle vertex $v_j$ of $H_n$ in only two cases: when $v_j$'s respective pendent vertex is already labeled or, when the colors $0$ or $3n$ (or both) have already been assigned to a vertex.
\end{lemma}

\begin{proof}
Let $H_n$, $n\geq 3$, be a helm graph and $v_j$ be a cycle vertex of $H_n$. Suppose $v_j$'s respective pendent vertex $v_{n+j}$ has not yet been labeled and none of the colors $0$ and $3n$ were assigned to a vertex. Moreover, consider it is Alice's turn and she labels $v_j$ with an arbitrary color $i$, $i\in \{0,\ldots,3n\}$. Note that Alice cannot choose $i\neq 0$ nor $i\neq 3n$ since once Alice labels $v_j$ with $i\in \{1,\ldots,3n-1\}$, on the next move Bob labels $v_{n+j}$ with $0$ or $3n$, winning the game.
\end{proof}

\begin{theorem}\label{thm:helms}
Bob has a winning strategy for all helms.
\end{theorem}

\begin{proof}
Given a helm graph $H_n$, $n\geq 3$, we first consider the case where Bob is the first player. He starts by labeling a pendent vertex $v_{n+j}$ with $0$, for $1\leq j\leq n$. According to Lemma~\ref{lemma:forceA}, Alice is forced to label $v_{j}$ with $3n$. Then, Bob assigns label $1$ to a vertex not adjacent to $v_{j}$. Since there are no possibilities left for Alice to create the edge label $3n-1$, Bob wins the game.

Now, suppose that Alice is the first player. By Lemma~\ref{lemma:Alabcyclevtx}, she has only two options for her first move: to label the center or a pendent vertex. Suppose Alice chooses to label the center with a arbitrary color $i$, $1\leq i\leq 3n-1$. If $i=1$ (resp. $i\neq 1$), then Bob labels a pendent vertex $v_{n+p}$, $1\leq p\leq n$, with $3n$ (resp. $0$). Alice is now forced to label $v_{p}$ with $0$ (resp. $3n$). In order to win, Bob labels a vertex not adjacent to $v_{p}$ with $3n-1$ (resp. $1$). Next, suppose that Alice chooses to label a pendent vertex $v_{n+j}$, $1\leq j\leq n$, with an arbitrary color $i$, $1\leq i\leq 3n-1$. Then, Bob assigns $0$ to a second pendent vertex $v_{n+p}$, $1\leq p\leq n$ and $p\neq j$. This way, Alice is forced to label $v_{p}$ with $3n$. If $i=1$, then Bob wins since there is no way for Alice to generate the edge label $3n-1$. If $i\neq 1$, Bob assigns label $1$ to a vertex not adjacent to $v_{p}$ and the result follows.
\end{proof}

A class similar to helms are the \emph{web graphs}, defined by Koh et al.~\cite{Kho1980} in 1980 as a graph obtained by connecting the pendent vertices $v_{n+1},v_{n+2},\ldots,v_{2n}$ of a helm into a cycle $(v_{n+1},v_{n+2},\ldots,v_{2n})$ and then linking a single new pendent vertex to each vertex of this outer cycle. Later, Kang et al.~\cite{Kang1996} extended the definition of web graphs so that the process of creating a new cycle by joining the pendent vertices and adding a pendent edge to each vertex of the outer cycle could be repeated as many times as desired. In this paper, we use the definition proposed by Kang et al.~\cite{Kang1996}, in which $W(t,n)$ denotes the web graph formed by $t$ vertex-disjoint $n$-cycles, where $t\geq 2$. These $t$ $n$-cycles are called the \emph{concentric cycles} of $W(t,n)$. A graceful labeling of the web graph $W(2,4)$ is exhibited in Figure~\ref{fig:web}.

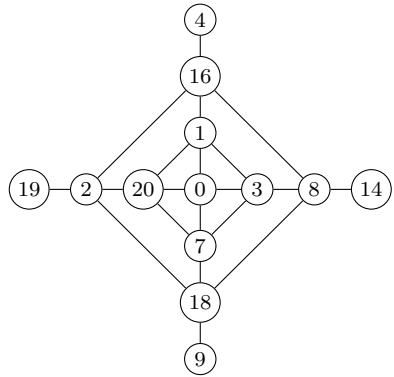
\begin{figure}[!htb]
\centering
\begin{tikzpicture}[scale=1.5]
\tikzstyle{every node}=[draw,shape=circle,minimum size=.3em,inner sep=1.8pt];
\path (0:0cm) node (v0) {\scriptsize $0$};
\path (90:0.5cm) node (v2) {\scriptsize $1$};
\path (0:0.5cm) node (v1) {\scriptsize $3$};
\path (3*90:0.5cm) node (v4) {\scriptsize $7$};
\path (2*90:0.5cm) node (v3) {\scriptsize $20$};
\path (90:1cm) node (v8) {\scriptsize $16$};
\path (0:1cm) node (v5) {\scriptsize $8$};
\path (3*90:1cm) node (v7) {\scriptsize $18$};
\path (2*90:1cm) node (v6) {\scriptsize $2$};
\path (90:1.5cm) node (v12) {\scriptsize $4$};
\path (0:1.5cm) node (v9) {\scriptsize $14$};
\path (3*90:1.5cm) node (v11) {\scriptsize $9$};
\path (2*90:1.5cm) node (v10) {\scriptsize $19$};
\draw (v0) -- (v1)
(v0) -- (v2)
(v0) -- (v3)
(v0) -- (v4)
(v1) -- (v5)
(v3) -- (v6)
(v4) -- (v7)
(v2) -- (v8)
(v8) -- (v12)
(v5) -- (v9)
(v6) -- (v10)
(v7) -- (v11)
(v1) -- (v2) -- (v3) -- (v4) -- (v1)
(v8) -- (v5) -- (v7) -- (v6) -- (v8);
\end{tikzpicture}
\caption{Web graph $W(2,4)$ with a graceful labeling.}
\label{fig:web}
\end{figure}

By the definition, a web graph $W(t,n)$ has $n(t+1)+1$ vertices and $m=n(2t+1)$ edges. We partition the vertex set of $W(t,n)$ into $t+2$ parts. The first part comprises only the central vertex $v_0$, also called \emph{center}. The second part comprises the pendent vertices, denoted by $v_1, v_2,\ldots , v_n$. The other $t$ parts are each one formed by the vertices that give rise to each concentric cycle. We denote $v_{n+1}, v_{n+2},\ldots, v_{2n}$ the vertices of the outer concentric cycle; $v_{2n+1}, v_{2n+2},\ldots , v_{3n}$ the vertices of the next concentric cycle, and so on, until the inner concentric cycle $v_{tn+1},v_{tn+2},\ldots ,v_{(t+1)n}$. We also consider that $v_k, v_{n+k}, v_{2n+k},\ldots,v_{tn+k},v_0$ are arranged linearly, for any $k \in \{1,\ldots,n\}$.

A result similar to that presented in Lemma~\ref{lemma:Alabcyclevtx} can be proved to web graphs and is presented as follows.

\begin{lemma}~\label{lemma:Alabelweb}
Given a web graph $W(t,n)$ with $m$ edges, Alice can label an outer cycle vertex $v_j$, $j \in \{n+1,\ldots,2n\}$, in only two cases: when $v_j$'s respective pendent vertex is already labeled or, when the colors 0 or $m$ (or both) have already been assigned to a vertex.\qed
\end{lemma}

Kang et al.~\cite{Kang1996} proved that $W(2,n)$, $W(3,n)$, $W(4,n)$ are graceful and, Abhyankar~\cite{Abhyankar2002} proved the same for $W(t,5)$, $5\leq t\leq 13$. Even though there are not many results for the gracefulness of the web graphs, we provide a strategy in which we guarantee that Bob is the winner on any web graph.

\begin{theorem}~\label{thm:web}
Bob has a winning strategy for all webs.
\end{theorem}

\begin{proof}
Let $W(t,n)$ be a web graph with $m$ edges. When Bob is the first player he assigns label $0$ to a pendent vertex $v_j$, for $j\in \{1,\ldots, n\}$. By Lemma~\ref{lemma:forceA}, Alice's only option is to label $v_{n+j}$ with $m$. Now, Bob assigns label $1$ to a vertex not adjacent to $v_{n+j}$ and he wins.

Now, suppose that Alice is the first player. By Lemma~\ref{lemma:Alabelweb} and by the symmetry of the graph, we have that Alice has only three options for her first move.
\textbf{Option 1:} Alice labels the center $v_0$ with an arbitrary color $i\in \{1,\ldots , m-1\}$. In this case, Bob assigns the color 0 to a pendent vertex $v_p$, $p\in\{1,\ldots,n\}$, forcing Alice to label $v_{n+p}$ with $m$. If $i=1$, then the game is over and Bob is the winner since there is no possibility left for Alice to create the edge label $m-1$. If $i\neq 1$, in order to win the game Bob labels any vertex that is not adjacent to $v_{n+p}$ with 1.
\textbf{Option 2:} Alice labels a pendent vertex $v_j$, $j\in\{1,\ldots,n\}$ with an arbitrary color $i\in \{1,\ldots , m-1\}$. In this case, Bob labels a second pendent vertex $v_p$, $p\in \{1,\ldots,n\}$ and $p\neq j$, with 0. Alice is now forced to label $v_{n+p}$ with $m$. If $i=1$ then the game is over and Bob wins. If $i\neq 1$, Bob labels any vertex that is not adjacent to $v_{n+p}$ with 1 in order to win.
\textbf{Option 3:} Alice labels a vertex $v_j$, $j\in\{2n+1,\ldots, (t+1)n\}$, with an arbitrary color $i\in \{1,\ldots , m-1\}$. In this case, Bob assigns the color 0 to a pendent vertex $v_p$, $p\in \{1,\ldots,n\}$, so that $v_j$ and $v_p$ are not arranged linearly. This forces Alice to label $v_{n+p}$ with $m$. If $i=1$ then the game is over and Bob wins. If $i\neq 1$, Bob must label any vertex that is not adjacent to $v_{n+p}$ with 1 and he also wins.
\end{proof}

An \emph{n-dimensional hypercube graph} $Q_n$, or just \emph{hypercube}, is defined recursively in terms of the cartesian product of two graphs as follows:
\begin{enumerate}[(i)]
 \itemsep0em
 \item $Q_1 = K_2$
 \item $Q_n= K_2\square Q_{n-1}$
\end{enumerate}

The hypercube graph $Q_3$ is exhibited in Figure~\ref{fig:hypercube} with a graceful labeling. Hypercube graphs have been much studied in graph theory and it is well known~\cite{Harary1988} that every $Q_n$ is bipartite, $n$-regular and has $|V(Q_n)|=2^n$ and $|E(Q_n)|=n2^{n-1}$. In 1981, Kotzig~\cite{Kotzig1981} proved that all hypercubes are graceful. In the next result, we characterize the graceful game for hypercubes.

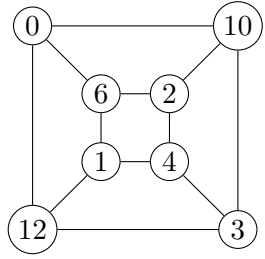
\begin{figure}[!htb]
\centering
\begin{tikzpicture}[main_node/.style={draw,shape=circle,minimum size=.3em,inner sep=1.9pt},scale=0.6]
    \node[main_node] (1) at (0,0) {12};
    \node[main_node] (2) at (1.5,1.5) {1};
    \node[main_node] (3) at (3,1.5) {4};
    \node[main_node] (4) at (4.5,0) {3};
    
    \node[main_node] (5) at (0,4.5) {0};
    \node[main_node] (6) at (1.5,3) {6};
    \node[main_node] (7) at (3,3) {2};
    \node[main_node] (8) at (4.5,4.5) {10};
    
    \draw (1) -- (2) -- (3) -- (4) -- (1);
    \draw (5) -- (6) -- (7) -- (8) -- (5);
    \draw (2) -- (6);
    \draw (3) -- (7);
    \draw (1) -- (5);
    \draw (4) -- (8);
\end{tikzpicture}
\caption{A graceful labeling of $Q_3$.}\label{fig:hypercube}
\end{figure}

\begin{theorem}\label{thm:hypercubes} Bob has a winning strategy for all \textit{hypercubes} $Q_n$ with $n\geq 2$.
\end{theorem}

\begin{proof} Let $Q_n$ be a hypercube with vertex set $V(Q_n)=\{v_1,v_2,\ldots,v_{2^n}\}$, $n\geq 2$. Since $Q_2\cong C_4$, the result for $Q_2$ follows from Theorem~\ref{thm:cyclegrp}. Thus, consider $Q_n$ with $n\geq 3$. Since $Q_n$ is bipartite, we partition $V(Q_n)$ into two disjoint sets $X$ and $Y$ such that
$X = \{v_j \colon 1\leq j \leq 2^{n-1}\}$ and $Y = \{v_j \colon 2^{n-1}+1 \leq j \leq 2^n\}$. 

When Bob is the first player, he starts by assigning $0$ to any vertex. Without loss of generality, suppose Bob chooses $v_1$. By Lemma~\ref{lemma:forceA}, Alice is forced to label a vertex adjacent to $v_1$ with $m = n2^{n-1}$. Without loss of generality, suppose that all $v_{2^{n-1}+j}$, for $1\leq j \leq n$, are adjacent to $v_1$. By the symmetry of the hypercube, we can also suppose w.l.o.g that Alice chooses $v_{2^{n-1}+1}$ to label. Now, Bob must assign label $1$ to a vertex adjacent to $v_1$, say $v_{2^{n-1}+2}$ (w.l.o.g). Alice is left with no other choice than to assign label $m-1$ to a vertex adjacent to $v_1$, say $v_{2^{n-1}+3}$ (w.l.o.g).

Since the edge labels $m$ and $m-1$ were already obtained, the next edge label Alice needs to create is $m-2$. Note that the vertex labels $1$ and $m-1$ are labeling two vertices in part $Y$, so it is impossible for Alice to use them in order to create the edge label $m-2$. Since Bob aims to lessen Alice's possibilities of using the pairs of values $0,m-2$ and $2,m$, he assigns label $2$ to $v_{2^{n-1}+4}$, forcing Alice to label $v_{2^{n-1}+5}$ with $m-2$. Note that this process can be repeated until there are no more free  vertices adjacent to $v_1$. 
If $n$ is even, then Bob guarantees his victory. Note that the last vertex adjacent to $v_1$ is labeled by Bob with the label $\frac{n}{2}$. That way, Alice cannot create the edge label $m-\frac{n}{2}$.
If $n$ is odd, then the last vertex adjacent to $v_1$ is labeled by Alice with the color $m-\left(\frac{n-1}{2}\right)$. The next edge label that Alice would need to generate is $m-\frac{n-1}{2}-1=m-\left(\frac{n-1}{2}+1\right)$. This edge label can be generated by the pair of values $i$ and $m-\left(\frac{n-1}{2}+1\right)+i$, for $0\leq i\leq \frac{n-1}{2}+1$. Moreover, note that the only pair of labels that is not labeling vertices on the same part is $m$ and $\frac{n-1}{2}+1$. Thus, Bob labels a vertex that is not adjacent to $v_{2^{n-1}+1}$ with $\frac{n-1}{2}+1$, winning the game.

Now, suppose that Alice is the first player. Without loss of generality, she starts by labeling $v_1\in X$ with an arbitrary label $i$. We consider four different cases depending on the value of $i$.

\textbf{Case 1:} $i=1$ (resp. $i=m-1$). On the next move, Bob labels a vertex $w\in Y$ that is not adjacent to $v_1$ with $0$ (resp. $m$). Then, Alice is forced to label a vertex $u\in X$ adjacent to $w$ with $m$ (resp. $0$). Now, Bob labels any vertex in $Y$ with $m-1$ (resp. $1$) exhausting Alice's possibilities of creating the edge label $m-1$.

\textbf{Case 2:} $2\leq i\leq \left\lfloor\frac{n}{2}\right\rfloor$. In this case, Bob assigns label $0$ to a vertex $w$ adjacent to $v_1$, $w\in Y$. Then Alice assigns label $m$ to a vertex $u$ adjacent to $w$, $u\in X$. From this moment on, Bob uses the same strategy for when he is the first player until reaches the point where Alice labels a  vertex adjacent to $w$ with label $m-(i-1)$. Now, it is sufficient for Bob to label any vertex in $Y$ with $m-i$. This way, Alice cannot create the edge label $m-i$, loosing the game.

\textbf{Case 3}: $\left\lfloor\frac{n}{2}\right\rfloor < i < m-\left\lfloor\frac{n}{2}\right\rfloor$. When $n$ is odd, Bob assigns label $0$ to a vertex $w \in Y$ adjacent to $v_1$. On the other hand, when $n$ is even, Bob assigns label $0$ to a vertex $w\in X$ that is not adjacent to $v_1$. Then, in the next turn, Alice is forced to assign label $m$ to a free neighbor of $w$  so as to generate edge label $m$. From this moment on, Bob applies the same  strategy for when he is the first player: in the next turn, Bob assigns label $1$ to a free neighbor of $w$, which forces Alice to assign label $m-1$ to another free neighbor of $w$ so as to generate edge label $m-1$, and so on, until the game reaches the point where Bob assigns label $\lfloor n/2 \rfloor$ to the last free neighbor of $w$. Bob's last move precludes Alice from obtaining the edge label $m-\lfloor n/2 \rfloor$ (note that the unique way left to obtain edge label $m-\lfloor n/2 \rfloor$ is by using the pair of labels $0$ and $m-\lfloor n/2 \rfloor$. However, every neighbor of $w$ is already labelled and none of them was assigned label $m-\lfloor n/2 \rfloor$). Therefore, Bob wins the game.

\textbf{Case 4:} $i=m-k$ for $2\leq k \leq \left\lfloor\frac{n}{2}\right\rfloor$. Bob assigns label $0$ to a vertex $w\in Y$ that is not adjacent to $v_1$. Alice is forced to assign label $m$ to a vertex $u$ adjacent to $w$, $u\in X$. From this moment on, Bob applies the same strategy he uses when he is the first player until reaches the point where Alice assigns label $m-(k-1)$ to a neighbor of $w$. Now, it is sufficient for Bob to label another vertex adjacent to $w$ with $k$. This way, it is not possible for Alice to create the edge label $m-k$. Therefore, Bob is the winner and the result follows.
\end{proof}

The \emph{prism} graph $P_{r,2}$, $r\geq 3$, is defined as the cartesian product $C_r \square P_2$ of a cycle on $r$ vertices and path $P_2$. Frucht and Gallian~\cite{Frucht1988} proved that every prism is graceful. Theorem~\ref{theorem:Prisms} characterizes the graceful game for all prisms.

\begin{theorem}\label{theorem:Prisms}
 Bob has a winning strategy for all prisms.
\end{theorem}

\begin{proof}
Let $P_{r,2}$ be a prism graph with vertex set $V(P_{r,2}) = \{v_{p,q} \colon 0 \leq p \leq r-1, 0 \leq q \leq 1\}$, where all vertices $v_{p,q}$ with the same index $q$ induce a cycle $C_r$.  Let $m = |E(P_{r,2})|$. First, consider $P_{r,2}$ with $r\geq 4$. When Bob is the first player, he starts by assigning label $0$ to vertex $v_{0,0}$, which forces Alice to assign label $m$ to a neighbor of $v_{0,0}$ (denote by $w$ the vertex she chooses.) In order to lessen Alice's possibilities of getting edge label $m-1$, Bob can assign label $m-1$ (resp. $1$) to a free vertex adjacent to $w$ (resp. $v_{0,0}$). Without loss of generality, for his next move, he assigns $m-1$ to a free neighbor of $w$, call it $u$. Now, Alice generates edge label $m-1$ by assigning label $1$ to the remaining free neighbor of $w$. With this configuration, the only way of Alice getting an edge label $m-2$ is by assigning label $m-2$ to a free neighbor of $v_{0,0}$ not adjacent to $u$ (there is at least one). However, since now it is Bob's turn, he assigns label $m-2$ to a vertex not adjacent to neither $v_{0,0}$ nor $u$ (there is at least one), thus winning the game.
 
Now, consider that Alice is the first player. She starts by assigning label $i \in \{1,2,\ldots,m-1\}$ to an arbitrary vertex of $P_{r,2}$. Adjust notation so that the vertex she chooses is the vertex $v_{0,0}$. We consider two different cases, depending on the value of label $i$.
 
 \textbf{Case 1:} $2 \leq i \leq m-2$. 
 Bob is the next player. He assigns label $0$ to a neighbor of $v_{0,0}$, call it $w$, thus forcing Alice to assign label $m$ to a free neighbor of $w$ so as to generate edge label $m$. 
 Alice can obtain the edge label $m-1$ from the pairs of vertex labels $m,1$ and $0,m-1$. However, now it is Bob's turn and he assigns label $1$ to the remaining free neighbor of $w$, thus precluding Alice from obtaining the edge label $m-1$.
 
 \textbf{Case 2:} $i\in\{1,m-1\}$. 
 We prove the case $i=1$ (the case $i = m-1$ is analogous by the complementary labeling.) Bob is the next player. He assigns label $0$ to a neighbor of $v_{0,0}$, call it $w$, thus forcing Alice to assign label $m$ to one of the free neighbors of $w$, call it $u$, so as to generate edge label $m$. By now, only one possibility of Alice getting edge label $m-1$ remained, which consists on assigning label $m-1$ to the remaining free vertex adjacent to vertex $w$. However, now it is Bob's turn and he assigns label $m-1$ to a remaining free neighbor of $v_{0,0}$ that is not adjacent to $u$, thus generating edge label $m-2$ and precluding Alice from obtaining a graceful labeling of the graph.

\smallskip

In order to complete the proof, it remains to analyze the graph $P_{3,2}$. Consider any graceful labeling $f$ of $P_{3,2}$ represented as a 6-tuple $(f(v_{0,0}),f(v_{1,0}),f(v_{2,0}),f(v_{0,1}),f(v_{1,1}),f(v_{2,1}))$. Figure~\ref{pic:P32graceful} shows all graceful labelings of $P_{3,2}$. 

\begin{figure}[!htb]
\centering
\begin{tabu} to 1\textwidth { | X[l] | X[l] | X[l] | X[l] | X[l] |}
 \hline (0, 9, 1, 4, 2, 7) & (0, 9, 2, 6, 1, 5) & (0, 9, 4, 8, 7, 1) & (0, 9, 5, 2, 1, 8) & (0, 9, 7, 8, 3, 4)\\
 \hline (0, 9, 8, 7, 5, 2) & (0, 1, 8, 9, 7, 4) & (0, 2, 5, 9, 8, 1) & (0, 5, 2, 9, 1, 8) & (0, 8, 1, 9, 4, 7)\\
 \hline
\end{tabu}
\caption{All graceful labelings of $P_{3,2}$.}
\label{pic:P32graceful}
\end{figure}

When Bob is the first player, he assigns label $0$ to an arbitrary vertex, say $v_{0,0}$. Then, Alice is forced to assign label $9$ to a neighbor of $v_{0,0}$ so as to generate edge label $9$. If she assigns $9$ to $v_{0,1}$ (that belongs to the $r$-cycle that does not contain $v_{0,0}$), then, in the next round, Bob can assign label $3$ or $6$ to any vertex in order to win the game (according to Figure~\ref{pic:P32graceful} the partial labeling obtained so far cannot be extended to a graceful labeling of $P_{3,2}$). On the other hand, if, in the second round, Alice assigns $9$ to $v_{1,0}$ or $v_{2,0}$, then Bob assigns label $6$ to a free vertex adjacent to the vertex labeled $9$ in order to win the game (according to Figure~\ref{pic:P32graceful} there is no graceful labeling of $P_{3,2}$ with adjacent vertices labeled $9$ and $6$).

When Alice is the first player, she assigns label $x \in \{1,\ldots,8\}$ to an arbitrary vertex, say $v_{0,0}$. Thus, in order to win the game in his first move, it suffices for Bob to assign to a neighbor of $v_{0,0}$ a label $\ell_x \in \{0,1,\ldots,9\}\backslash\{x\}$ that is not adjacent to $x$ in any graceful labeling of $P_{3,2}$. According to the graceful labelings exhibited in Figure~\ref{pic:P32graceful}, for every $x\in \{1,\ldots,8\}$, $\ell_x$ does exist, as follows: $\ell_1 = 3$, $\ell_2 \in \{3,6\}$, $\ell_3 \in \{0,1,2,5,6,7\}$, $\ell_4\in \{5,6\}$, $\ell_5\in\{3,4\}$, $\ell_6\in\{2,3,4,7,8,9\}$, $\ell_7\in\{3,6\}$, $\ell_8=6$, and the result follows. 
\end{proof}

Given a simple graph $G$ and two distinct vertices $u,v \in V(G)$, the \emph{distance} between $u$ and $v$ in $G$ is the number of edges in a shortest path between $u$ and $v$, and is denoted by $d_G(u,v)$. If there is no such a path, then $d_G(u,v)= \infty$. The \emph{k-th power} of a simple graph $G$ is the simple graph $G^k$ that has vertex set $V(G^k)=V(G)$, with distinct vertices $u,v$ being adjacent in $G^k$ if and only if $d_G(u,v)\leq k$. 

The family of \emph{powers of paths} comprises all graphs $G^k$ obtained when $G \cong P_n$, $n\geq 1$. The $k$-th power of a path $P_n$ is denoted by $P_n^k$. It is known that all powers of paths $P_n^2$ are graceful~\cite{Kang1996} and, in the next theorem, the graceful game is characterized for all $P_n^2$.

\begin{theorem}\label{thm:pathpot} Bob has a winning strategy for all $P_n^2$ with $n\geq 4$. Alice wins on $P_3^2$.
\end{theorem}

\begin{proof}
Let $P_n^2$ be the 2nd-power of a path $P_n$ on $n\geq 3$ vertices. The linear sequence $(v_0,v_1,\ldots$, $v_{n-1})$ that composes $P_n$ is also a linear sequence of $P_n^2$ since $P_n \subset P_n^2$. Let $m = |E(P_n^2)| = 2n-3$. Since  $P_1^2\cong P_1$, $P_2^2\cong P_2$ and $P_3^2\cong C_3$, the result for these graphs follows from Theorems~\ref{thm:labpath} and~\ref{thm:cyclegrp}. Thus, assume that $n\geq 4$.

Consider that Bob is the first player. He starts by assigning label $0$ to $v_0$, giving Alice the options of labeling $v_1$ or $v_2$ with $m$. 
First, suppose that Alice chooses the first option. Then, Bob labels $v_3$ with $m-1$ on his next move, creating the edge label $1$. It is now impossible for Alice to use the labels $0$ and $m-1$ to create the edge label $m-1$. Her only option would be to use $m$ and $1$ on adjacent vertices and, to do so, she would have to label $v_2$ with $1$. However, this move is forbidden by the rules of the game since it would create a second edge label $1$. Therefore, Bob wins.

Now, suppose that Alice chooses the second option. Then, Bob labels $v_3$ with $\frac{m+1}{2}$ ($m$ is odd for all $n$).\footnote{When $n=4$, Bob wins the game at this point since no label can be assigned to $v_1$ without generate repeated edge labels.} This way, Alice cannot label either $v_1$ or $v_4$ with $1$ since it would generate a repeated edge label $\frac{m-1}{2}$. Thus, her only option to create the edge label $m-1$ is to label $v_1$ with $m-1$.\footnote{When $n=5$, Bob wins the game at this point since no label can be assigned to $v_4$ without generate repeated edge labels.} With the edge labels $m$ and $m-1$ guaranteed, Bob must try to prevent Alice from creating the edge label $m-2$. To achieve this, it is enough to label $v_4$ with any label other than $2$.

\medskip

Next, consider that Alice is the first player. In this case, the proof is divided into seven different cases. We start with the general case and end with the particular cases since we use the general case strategy in some parts of some particular cases. Suppose that Alice labels a vertex $v_j$ with an arbitrary label $i \in \{1,\ldots,m-1\}$ (by Lemma~\ref{lemma:firstlab0}). By the symmetry of $P_n^2$ along its linear sequence $(v_0,v_1,\ldots,v_{n-1})$, we have that starting by labeling $v_j$ is analogous to start by labeling $v_{n-j-1}$. Thus, consider $j\in \{\lfloor n/2 \rfloor,\ldots,n-1\}$. 

\smallskip

\textbf{Case 1:} $P_n^2$ with $n\geq 10$. Suppose that $i=\frac{m+1}{2}$ or $i=1$. Bob must label $v_0$ with $m$, leaving Alice with only two options: to assign 0 to $v_1$ or to assign 0 to $v_2$. If she chooses the first option, then it suffices Bob to label $v_2$ with $\frac{m-1}{2}$. The only option Alice has to create the edge label $m-1$ would be by assigning $m-1$ to $v_3$. However, this move would generate two edges with the same label $|f(v_1)-f(v_2)|=\frac{m-1}{2}=|f(v_2)-f(v_3)|$. Hence, Bob wins since Alice cannot create the edge label $m-1$. In case Alice chooses the second option, Bob must label $v_3$ with $\frac{m-1}{2}$, making it impossible for Alice to assign the color $m-1$ to $v_4$ in order to create the edge label $m-1$. If $i=1$, then Alice cannot use label $1$ to label $v_1$. The game ends with Bob being the winner. If $i=\frac{m+1}{2}$, then Alice's only option is to label $v_1$ with 1, creating not only the edge label $|f(v_0)-f(v_1)|=m-1$ but also $|f(v_1)-f(v_3)|=\frac{m-3}{2}$. Now, the only way to create the edge label $m-2$ would be to assign $m-2$ to $v_4$. However, this move would generate the repeated edge label $\frac{m-3}{2} = |f(v_3)-f(v_4)|$. Hence, Bob is the winner.

Now, suppose that $i\neq \frac{m+1}{2}$ and $i\neq 1$. Thus, Bob labels $v_0$ with $0$, leaving Alice with no other options than to label $v_1$ or $v_2$ with $m$. If she chooses to assign $m$ to $v_1$, Bob simply labels $v_2$ with $\frac{m+1}{2}$. The only option Alice has to create the edge label $m-1$ is to assign $1$ to $v_3$. However, this move induces two repeated edge labels $|f(v_1)-f(v_2)|=\frac{m-1}{2}=|f(v_2)-f(v_3)|$, which is forbidden. In case Alice chooses to assign $m$ to $v_2$, Bob labels $v_3$ with $\frac{m+1}{2}$, making it impossible for Alice to label $v_4$ or $v_1$ with $1$ in order to create the edge label $m-1$. If $i\in \{\frac{m+1}{2}-1, \frac{m+1}{2}+1, m-1\}$, then Alice cannot label $v_1$ with $m-1$ and Bob wins the game. If $i\not\in \{\frac{m+1}{2}-1, \frac{m+1}{2}+1, m-1\}$, then she is forced to label $v_1$ with $m-1$, creating the edge labels $|f(v_0)-f(v_1)|=m-1$ and $|f(v_1)-f(v_3)|=\frac{m-3}{2}$. Now, the only way to create the edge label $m-2$ would be to assign $2$ to $v_4$. However, this move is prohibited since it generates the repeated edge label $\frac{m-3}{2} = |f(v_3)-f(v_4)|$. Hence, Bob wins.

\smallskip

\textbf{Case 2:} $P_4^2$. First, consider $P_4^2$ with $j=2$ and $f(v_j) = i \in \{1,2,3,4\}$. Consider a graceful labeling $f$ of $P_4^2$ represented as a 4-tuple $(a_0,a_1,a_2,a_3)$ where $a_j = f(v_j)$, for $0\leq j\leq 3$. Thus, all graceful labelings of $P_4^2$ with $j=3$ and $i\in \{1,2,3\}$ are: 
$(5,0,1,3)$, $(3,0,1,5)$, $(1,5,2,0)$, $(0,5,2,1)$, $(5,0,3,4)$, $(4,0,3,5)$, $(2, 5, 4, 0)$ and $(0, 5, 4, 2)$. Therefore, for any move made by Alice, Bob labels $v_1$ with $1$ in order to guarantee that a graceful labeling of the graph cannot be obtained. For the case where $j=3$ and $f(v_j) = i \in \{1,2,3,4\}$, all graceful labelings of $P_4^2$ are: 
\begin{align*}
(0, 2, 5, 1)&\quad (2, 0, 5, 1)\quad (2, 5, 0, 1)\quad (3, 0, 5, 1)\quad (3, 5, 0, 1)\quad (5, 0, 2, 1) \\ 
(0, 4, 5, 2)&\quad (0, 5, 4, 2)\quad (1, 0, 5, 2)\quad (1, 5, 0, 2)\quad (4, 0, 5, 2)\quad (4, 5, 0, 2) \\
(1, 0, 5, 3)&\quad (1, 5, 0, 3)\quad (4, 0, 5, 3)\quad (4, 5, 0, 3)\quad (5, 0, 1, 3)\quad (5, 1, 0, 3) \\
(2, 0, 5, 4)&\quad (2, 5, 0, 4)\quad (3, 0, 5, 4)\quad (3, 5, 0, 4)\quad (5, 0, 3, 4)\quad (5, 3, 0, 4).
\end{align*}

Thus, in order to preclude the graph from being gracefully labeled, Bob must: (i) to assign label 4 to any free vertex if $i=1$; (ii) to assign label $3$ to any free vertex if $i=2$; (iii) to assign label $2$ to any free vertex if $i=3$; or (iv) to assign label $1$ to any free vertex if $i=4$.

\smallskip

\textbf{Case 3:} $P_5^2$. \emph{Subcase 3.1:} $j=2$ and $i=1$. Since $v_3$ is adjacent to all other vertices of the graph, Bob simply labels any other vertex with $2$, creating the edge label 1. This move makes it impossible to assign 0 to any other vertex, because it would generate other edge label 1. Since the only way of creating the edge label 7 is by using vertex labels 0 and 7, the actual labeling cannot be extended to a graceful labeling of the graph.

\emph{Subcase 3.2:} $j=2$ and $i=6$. By the same reasoning as the previous subcase, Bob labels an arbitrary free vertex with 5 in order to make it impossible to assign 7 to any other vertex. Thus, Bob wins since there is no way of creating the edge label 7.

\emph{Subcase 3.3:} $j=2$, $i\neq 1$ and $i\neq 6$. Bob labels $v_0$ with 7, forcing Alice to label $v_1$ with 0.
If $i=3$, Bob wins the game since there is no way to generate the edge label $6$ through a valid move. If $i=2$, the only way to create the edge label 6 is by labeling $v_3$ with 6. However, this move is not allowed since it generates two edges labels $4$ ($v_0v_2$ and $v_2v_3$). In this case, Bob is the winner. For $i\in \{4,5\}$, Bob assigns $6$ to $v_4$, making it impossible for Alice to create the edge label 6 (The reader can verify that such assignment is a valid move).

\emph{Subcase 3.4:} $j=3$ and $i\neq 6$. Bob labels $v_0$ with 7, leaving Alice with two choices: assigning $0$ to $v_1$ or to $v_2$. If Alice chooses the first one, then we leave to the reader to verify that, for each $i$, there is at least one $k$, $2\leq k\leq 5$, such that Bob can assign it to $v_2$ without breaking any rules. This way, Alice cannot create the edge label $6$. Hence, Bob wins. If Alice chooses the second choice, then Bob must label $v_4$ with $i+1$ or $i-1$, creating the edge label  $1=|f(v_3)-f(v_4)|$. This way, if Alice labels $v_1$ with $1$ or $6$ on the attempt of creating the edge label $6$, she would also be creating a repeated edge label $1$. Hence, Bob is the winner.

\emph{Subcase 3.5:} $j=3$ and $i=6$. Bob labels $v_0$ with 0, leaving Alice with the only options of labeling $v_1$ or $v_2$ with 7. For the first case, Bob labels $v_4$ with 1 and exhausts Alice's possibilities of creating the edge label 6. For the second case, Bob labels $v_1$ with 2 creating the edge label $5=|f(v_1)-f(v_2)|$. Now, if Alice labels $v_4$ with 1 in order to create the edge label $6=|f(v_2)-f(v_4)|$, she would also be creating $5=|f(v_3)-f(v_4)|$, which is forbidden. Since she cannot create an edge label 6, Bob is the winner.

\emph{Subcase 3.6:} $j=4$ and $i=6$. Bob labels $v_0$ with 0, leaving Alice with the only two options of assigning 7 to $v_1$ or $v_2$. If Alice chooses the first option, it suffices Bob to label $v_3$ with 5, creating the edge label $1=|f(v_3)-f(v_4)|$. This way, Alice cannot use the colors 0 and 6 to create the edge label 6 and, the only way to use the colors 7 and 1 to do so is by labeling $v_2$ with 1. However, this move would generate the repeated edge label $1=|f(v_0)-f(v_2)|$, which is not allowed. On the other hand, if Alice chooses the second option, Bob labels $v_1$ with 5, creating $|f(v_0)-f(v_1)|=5$. Now, the only way Alice could create the edge label 6 would be by assigning 1 to $v_3$, but this move would generate the repeated edge label $5 = |f(v_3)-f(v_4)|$, which is forbidden.

\emph{Subcase 3.7:} $j=4$ and $i=3$. Bob labels $v_0$ with 7, forcing Alice to label either $v_1$ or $v_2$ with 0. In the first case, Bob labels $v_2$ with 4, creating the edge label $3 = |f(v_0)-f(v_2)|$. This way, for Alice to create the edge label 6, she would have to assign 6 to $v_3$. However, this move creates $|f(v_3)-f(v_4)|=3$, which is against the rules. Hence, Bob wins. In the second case, Bob labels $v_3$ with 4, creating $|f(v_3)-f(v_4)|=1$. Now, for Alice to create the edge label 6, she has to label $v_1$ with 1 or 6. However, she cannot label $v_1$ with 1 because it would generate a repeated edge label $3=|f(v_1)-f(v_3)|$ and an edge with this label already exists ($v_2v_4$). This leaves Alice with the only option of labeling $v_1$ with 6, but this move is not allowed since it generates the repeated edge label $|f(v_0)-f(v_1)|=1$. 

\emph{Subcase 3.8:} $j=4$, $i\neq 3$ and $i\neq 6$. Bob labels $v_0$ with 7, forcing Alice to assign 0 to either $v_1$ or $v_2$. If Alice chooses the first one, Bob labels $v_2$ with 3. This way, Alice cannot create the edge label 6 because she would have to label $v_3$ with 6 and this move would generate two edge labels 3 ($v_1v_2$ and $v_2v_3$). Suppose now that Alice chooses the second option. If $i=4$, Bob labels $v_1$ with 2, creating $|f(v_1)-f(v_2)|=2$. This way, Alice cannot label $v_3$ with 6 because it would generate a second edge label 2 ($v_3v_4$). If $i\neq 4$, Bob must label $v_1$ with 3, making it impossible for Alice to label $v_3$ with 6.

\smallskip 

\textbf{Case 4:} $P_6^2$. \emph{Subcase 4.1:} $j=3$ and $i=1$ (resp.~$i=8$). Bob labels $v_0$ with $9$ (resp.~$0$), forcing Alice to label either $v_1$ or $v_2$ with $0$ (resp.~$m$). In any of these two cases, Bob labels $v_5$ with $8$ (resp.~$1$) guaranteeing that Alice cannot create the edge label $8$.

\emph{Subcase 4.2:} $j=3$, $i\neq 1$ and $i\neq 8$. Bob labels $v_0$ with 0. Thus, Alice is forced to label $v_1$ (resp.~$v_2$) with $9$. It is not difficult to verify that, for each value of $i$, there is at least one label $k$, $2\leq k \leq 7$, that Bob can assign to $v_2$ (resp. $v_1$) without breaking any rules and this move prevents Alice from creating the edge label $8$.

\emph{Subcase 4.3:} $j=4$ and $i\in \{1,8\}$. If $i=1$ (resp. $i=8$), then Bob labels $v_0$ with $9$ (resp. 0). This way, Alice is forced to label $v_1$ or $v_2$ with 0 (resp. $9$). Then, it suffices Bob to label $v_3$ with $4$ (resp. $5$). Therefore, Alice cannot label either $v_1$ or $v_2$ with 1 (resp. $8$), because this color is already assigned to $v_5$. She also cannot label neither of these two vertices with $8$ (resp. 1), because it would generate two edge labels $4 = |f(v_1)-f(v_3)| = |f(v_2)-f(v_3)|$. Hence, Bob is the winner. 

\emph{Subcase 4.4:} $j=4$, $i\neq 1$ and $i \neq 8$. Bob labels $v_0$ with 0. Alice is thus forced to label $v_1$ or $v_2$ with $9$. Independently of her choice, Bob labels $v_5$ with $i-1$, thus generating edge label $1 = |f(v_4)-f(v_5)|$. From now on, any Alice's attempt to obtain edge label $m-1$ generates a repeated edge label $1$. Therefore, Bob wins. 

\emph{Subcase 4.5:} $j=5$. This case applies to Case 1.

\smallskip

\textbf{Case 5:} $P_7^2$. For $j=3$ and $j=4$, the development of the cases of $P_7^2$ are analogous to the cases of $P_6^2$ (for their respective $j$'s). The cases $j=5$ and $j=6$ are analogous to Case 1.

\smallskip 

\textbf{Case 6:} $P_8^2$. For $j=4$, the development of the cases of $P_8^2$ are analogous to the cases of $P_6^2$. The cases $j=5$, $j=6$ and $j=7$ apply to Case 1.

\smallskip 

\textbf{Case 7:} $P_9^2$. For $j=4$, the development of the cases of $P_9^2$ are analogous to the cases of $P_6^2$. The cases $j=5$, $j=6$, $j=7$ and $j=8$ apply to Case 1.
\end{proof}

\section{Concluding Remarks}\label{sec:conclusion}

In this work, we investigate the graceful game introduced by Tuza \cite{Tuza2017} for many classic families of graphs in order to contribute to the two famous graceful graph conjectures posed by Rosa \cite{Rosa1967,Rosa1988}.
As we summarize in Table~\ref{tabela}, Alice has winning strategies for only few cases such as: complete graphs $K_i$, $i\leq 3$, and stars $K_{1, q}$, $q\geq 2$ when she is the first player.

\begin{table}[ht]
\begin{center}
\begin{tabular}{ccc}
\hline
\multicolumn{1}{c}{\multirow{2}{*}{Graph class}} & \multicolumn{2}{c}{First player}  \\ \cline{2-3}
\multicolumn{1}{c}{}
                                & Alice & Bob \\ \hline \hline
$P_n, n={1,2}$                  & A     & A   \\ \hline
$P_3$                           & A     & B   \\ \hline
$P_n, n\geq 4$                  & B     & B   \\ \hline
$K_3$                           & A     & A   \\ \hline
$K_4$                           & B     & B   \\ \hline
$C_n, n\geq 4$                  & B     & B   \\ \hline
$K_{1,q}, q\geq 2$              & A     & B   \\ \hline
$K_{p,q}, p,q\geq 2$            & B     & B   \\ \hline
$cat(k_1,\ldots ,k_s), s\geq 2$ & B     & B   \\ \hline
$W_n, n={3,4,5}$                & B     & B   \\ \hline
$W_n, n\geq 6$                  & ?     & B   \\ \hline
$H_n, n\geq 3$                  & B     & B   \\ \hline
$W(t,n), t\geq 2, n\geq 3$      & B     & B   \\ \hline
$G_n, n\geq 3$                  & B     & B   \\ \hline
$Q_n, n\geq 2$                  & B     & B   \\ \hline
$P_n^2, n\geq 4$                & B     & B   \\ \hline
\end{tabular}
\end{center}
\caption{\label{tabela} Graph classes and winners: A (Alice) and B (Bob).}
\end{table}

\section{Acknowledgements}
This work was partially supported by CNPq, FAPERJ, Coordena\c{c}\~ao de Aperfei\c{c}oamento de Pessoal de N\'ivel Superior - Brasil (CAPES) - Finance Code 001 and CAPES-PrInt project number $88881.310248$/$2018$-$01$.

\end{document}